\documentclass[11pt,a4paper,headsepline,footsepline]{scrartcl}
\usepackage[utf8]{inputenc} % Required for inputting international characters
\usepackage[TU]{fontenc} % Output font encoding for international characters
\usepackage[english]{babel}

\usepackage{scrlayer-scrpage}
\clearpairofpagestyles
\ohead{\textsc{Page \pagemark}}

\usepackage{authblk}
\usepackage[colorlinks,citecolor=blue,linkcolor=blue]{hyperref}
\usepackage{todonotes}
\usepackage{mathtools}
\usepackage{amsthm}
\usepackage{mathrsfs}
\usepackage{amsmath}
\usepackage{amsfonts}
\usepackage{amssymb}

\newtheorem{remark}{Remark}[section]

\newtheorem{example}{Example}
\newtheorem{theorem}{Theorem}[section]

\newtheorem{Prop}{Proposition}[section]

\usepackage{graphicx}
\usepackage{tikz}
\usepackage{color}
\usetikzlibrary{arrows.meta}
\usetikzlibrary{positioning}
\usepackage{booktabs}
\usepackage{subcaption}

\usepackage[autostyle=true, babel=true]{csquotes} 
\usepackage[backend=biber,natbib=false,style=numeric,sortcites]{biblatex} % Use the bibtex backend with the authoryear citation style (which resembles APA)
\renewbibmacro{in:}{}
\bibliography{lienard.bib} % The filename of the bibliography

\newcommand{\reeb}{\mathcal{R}} %Reeb
\newcommand{\cH}{\mathcal{H}} %Hamiltonian Function
\newcommand{\vH}{X_{\cH}} %Hamiltonian vector Field
\newcommand{\vecq}{\frac{\partial}{\partial q^i}} %Canonical basis for TM
\newcommand{\vecp}{\frac{\partial}{\partial p_i}}
\newcommand{\vecs}{\frac{\partial}{\partial s}}
\newcommand{\ch}{h}
\newcommand{\cO}{\mathcal{O}}

\usepackage{authblk}
\title{Geometric numerical integration of Li\'enard systems via a contact Hamiltonian approach}

\newbox{\myorcidaffilbox}
\sbox{\myorcidaffilbox}{\large\includegraphics[height=1.7ex]{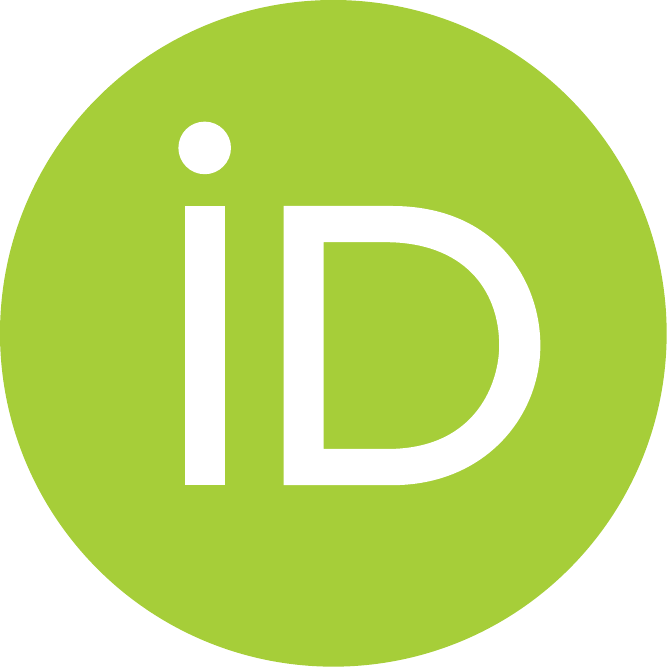}}
\newcommand{\orcidaffil}[1]{%
	\href{https://orcid.org/#1}{\usebox{\myorcidaffilbox}}}

\author[1]{Federico Zadra \orcidaffil{0000-0001-7565-3792}}
\author[2]{Alessandro Bravetti \orcidaffil{0000-0001-5348-9215}}
\author[3]{Marcello Seri \orcidaffil{0000-0002-8563-5894}}
\affil[1]{Bernoulli Institute for Mathematics, Computer Science and Artificial Intelligence, 
\authorcr\small Groningen, The Netherlands,
\authorcr\normalsize \texttt{f.zadra@rug.nl}}
\affil[2]{Instituto de Investigaciones en Matem\'aticas Aplicadas y en Sistemas (IIMAS--UNAM), \authorcr\small Mexico City, Mexico,
 \authorcr\normalsize \texttt{alessandro.bravetti@iimas.unam.mx}}
\affil[3]{Bernoulli Institute for Mathematics, Computer Science and Artificial Intelligence, 
\authorcr\small Groningen, The Netherlands,
\authorcr\normalsize \texttt{m.seri@rug.nl}}

\date{}
\begin{document}

\maketitle
\begin{abstract}
Starting from a contact Hamiltonian description of Li\'enard systems, we introduce  
a new family of explicit geometric integrators for these nonlinear dynamical systems. 
Focusing on the paradigmatic example of the van der Pol oscillator, we demonstrate that these integrators are particularly stable and preserve the qualitative features of the dynamics, even for relatively large values of the time step and in the stiff regime.


\end{abstract}

\section{Introduction}

Li\'enard systems are a class of 2-dimensional nonlinear dynamical systems that exhibit a stable limit cycle. Among them the most famous is the van der Pol oscillator~\cite{lienard1928,vanderpol1920}.
Due to the existence of a stable limit cycle, such systems are of the utmost importance in modelling natural phenomena such as e.g.~electrical circuits and neuronal dynamics,
and therefore an accurate investigation of their dynamics is required.
However, because of the nonlinear nature of such systems, analytical results are scarce and one has to recur to perturbative techniques and numerical integration.

An immediate and paramount problem for both the development of perturbative techniques and of stable numerical schemes is the lack of a geometric structure. 
Indeed, apart from very specific cases in which some integrability conditions 
are satisfied, and where one can use the Jacobi Last Multiplier to find a Lagrangian or Hamiltonian structure~\cite{nucci2010lagrangians,carinena2019nonstandard}, in the general case such pursuit is hopeless.
For instance,
many Li\'enard systems present an attractor, a stable limit cycle, and thus they cannot be Hamiltonian in the symplectic sense.
There have been several attempts in the literature in order to circumvent this problem. In~\cite{choi1973extended} the authors suggested to enlarge the phase--space to
a 4-dimensional manifold and define a particularly simple Hamiltonian system in this enlarged space so that the 2-dimensional projection onto the original space
recovers the original dynamics, and then they showed that this approach allows for the use of perturbative methods.
In~\cite{Shah2015}, the classical Bateman trick for the harmonic oscillator has been extended to the van der Pol oscillator and then further generalised 
to all Li\'enard systems with a quadratic potential.
Both these approaches involve a 4-dimensional phase--space and in both the authors have focused on the perturbation theory 
and have not explored the consequences of the Hamiltonisation for the numerical integration.
From yet another perspective, in~\cite{ChRaSt2018} the authors have presented various splitting schemes for ``conditionally linear systems'' (these include Li\'enard systems) 
which, although not being geometric, are based on the standard splitting schemes for symplectic Hamiltonian systems, and showed good qualitative and quantitative results.

In this work we contribute to the advancement of geometric integration for Li\'enard systems by using Hamiltonian flows on contact manifolds.
Contact geometry was introduced in Sophus Lie's study of differential equations, and has been the subject of an intense research, especially
related to low-dimensional topology~\cite{geiges2008introduction}.
In recent years, contact Hamiltonian systems have found many applications, 
first in the context of thermodynamics~\cite{Mrugala_1990, vdsm2018, bravetti2019contact} and, more recently, in the context of the Hamiltonisation of several dissipative dynamical 
systems~\cite{bravetti2016thermostat,Bravetti2017,Bravetti2020,vermeeren2019contact,gaset2019new,gaset2020contact, ciaglia2018contact}. 
The large number of applications of contact systems that have appeared recently motivated research on geometric numerical integration~\cite{Bravetti2020,simoes2020geometry,vermeeren2019contact}.
Fortunately, contact flows possess geometric integrators (both variational and Hamiltonian) that precisely parallel their symplectic counterparts, and
therefore they show remarkable numerical and analytical properties such as e.g.~increased stability, near-preservation of invariant quantities, and modified Hamiltonians.

In this work, leveraging some of the ideas in~\cite{choi1973extended}, we start a treatment of Li\'enard systems from the point of view of contact Hamiltonian systems:
we show that they can be given a particularly simple Hamiltonian formulation on a 3-dimensional
contact manifold, and then we use this Hamiltonisation to construct splitting integrators for such systems and analyse their properties from an analytical point of view, exploiting the modified equations.
Along the work we use the van der Pol oscillator as a paradigmatic example.

Our results show that the resulting geometric integrators are very stable, even when the system is stiff, and they preserve the qualitative features of the limit cycle even for large values of 
the time step, which permits to spare computational resources and is of primal importance in applications to e.g.~neuronal dynamics~\cite{ChRaSt2018}.
Moreover, from the use of the modified equations, we can prove analytical results on the preservation and the period of the limit cycle that show a very good agreement with the numerical simulations.

The paper is organised as follows:
in Section~\ref{sec:CHLienard}
we provide a Hamiltonian formulation of Li\'enard systems based on contact Hamiltonian dynamics, and then
in Section~\ref{sec:GeomInt}
we introduce a new class of explicit geometric integrators for these systems that are naturally derived by splitting the Hamiltonian.
Then in Sections~\ref{sec:numericalandanalytical}
and~\ref{sec:forcedvdP}
we thoroughly analyse the properties of these integrators both analytically and numerically by investigating the benchmark example of the van der Pol oscillator. We conclude in Section~\ref{sec:conclusions} with a discussion and a perspective on future work.

All the simulations are reproducible with the code provided in \cite{Zadra2020}.

\section{A contact Hamiltonian formulation of Li\'enard systems}\label{sec:CHLienard}

\subsection{A brief review of Li\'enard Systems}
\label{sec:reviewLi\'enardsystems}
Li\'enard systems are a family of planar coupled differential equations of the form~\cite{lienard1928}
	\begin{equation}
	\begin{cases}
	\dot{x} = y - F(x) \\
	\dot{y} = - g(x)
	\end{cases}, \label{Li\'enardsyst}
	\end{equation}
where $F(x)$ is the antiderivative of an even function $f(x)$ and $g(x)$ is an odd function.
Alternatively, \eqref{Li\'enardsyst} is equivalent to the second order scalar equation
	\begin{equation}
	\ddot{x}=- f(x) \dot{x}-g(x). \label{Li\'enard}
	\end{equation}
A third equivalent version of~\eqref{Li\'enardsyst} is given by
	\begin{equation}
	\begin{cases}
	\dot{x}=y,\\
	\dot{y}=-g(x) - f(x) y\,.
	\end{cases} \label{eq:Li\'enard2d}
	\end{equation}

	\begin{example}[The van der Pol oscillator]\label{ex:vdP}
	Perhaps the most famous example of the family of Li\'enard systems is the van der Pol
	oscillator, which can be written using dimensionless variables as follows
		\begin{equation}
		\ddot{x}=\epsilon(1-x^2)\dot{x}-x\,, \label{vanderpol}
		\end{equation}
	and can be equivalently rewritten in the form~\eqref{eq:Li\'enard2d} as
		\begin{equation}\label{eq:vdp2}
		\begin{cases}
		\dot{x}=y,\\
		\dot{y}=-x + \epsilon(1-x^{2}) y\,,
		\end{cases}
		\end{equation}
	from which we recognise that in this case $f(x)=-\epsilon(1-x^{2})$ and $g(x)=x$.
	\end{example}

A crucial property of Li\'enard systems is encoded in the following theorem, guaranteeing the existence and uniqueness of a stable limit cycle for a large class of systems~\cite{Perko1991}.
	\begin{theorem}[Li\'enard's Theorem]
	Under the conditions
		\begin{itemize}
		\item $F, g \in C^1(\mathbb{R})$,
		\item $xg(x)>0$ if $x \neq 0$,
		\item $F(0) = 0$ and $f(0)<0$,
		\item $F(x)$ has exactly one positive zero at $x=a$, is monotone increasing for $x>a$ and $\displaystyle\lim_{x\to+\infty}F(x) = +\infty$; 
		\end{itemize}
	the dynamical system \eqref{Li\'enardsyst} presents a unique, stable limit cycle. 
	\end{theorem}

In particular, the theorem above implies that the van der Pol equation \eqref{vanderpol} with $\epsilon>0$ has a unique, stable limit cycle. 

For additional information on the classical approach to the analysis of Li\'enard systems we refer to~\cite{Perko1991}.

\subsection{A brief review of contact Hamiltonian systems}
\label{sec:contactgeometryandintegrators}
Similarly to the fact that a symplectic manifold is a $2n$-dimensional differentiable manifold endowed with a 2-form $\omega$ that is closed ($d\omega=0$) and non-degenerate ($\omega^{n}\neq0$),
\emph{an exact contact manifold} $M$ is a $(2n+1)$-dimensional manifold endowed with a 1-form $\eta$, called \emph{the contact form}, that is non-degenerate, which means 
\begin{equation}
\eta \wedge (d\eta)^n\neq0.
\end{equation}

A contact version of Darboux's theorem~\cite{Aroldmechanics} guarantees the local existence of coordinates $(q^i,p_i,s)$ -- called \emph{Darboux coordinates} -- 
which permit to express the contact form
as $\eta = ds - p_i dq^i$, where Einstein's summation convention over repeated indices is being used here and in the following.

The contact form allows us to define in a natural way the concept of a Hamiltonian vector field on $M$. 
Let $\cH$ be a real function on $M$, then \emph{the contact Hamiltonian vector field $\vH$ associated with $\cH$} is defined by
	\begin{equation}
	\iota_{\vH} \eta = -\cH \qquad \iota_{\vH} d \eta = d\cH - \left(\iota_\reeb \cH \right) \eta\,,
	\end{equation}
where $\iota_{X}$ is the interior product and $\reeb$ is the Reeb vector field corresponding to $\eta$ \cite{geiges2008introduction}.

In Darboux coordinates $\vH$ takes the form
\begin{equation}
	\vH =\underbrace{ \left(\frac{\partial \cH}{\partial p_i} \right)}_{\dot{q}} \vecq + \underbrace{\left(-p_i \frac{\partial \cH}{\partial s} - 
	\frac{\partial \cH}{\partial q^i} \right)}_{\dot{p}} \vecp + \underbrace{\left(p_i \frac{\partial \cH}{\partial p_i} - \cH \right)}_{\dot{s}} \vecs, \label{eq:hamvecfield}
\end{equation}
Finally, contact manifolds carry a natural bracket structure, called \emph{the Jacobi bracket}, which yields a Lie algebra on smooth functions on $M$
and is defined as
\begin{equation}
\left\{f,g\right\}_\eta:=-\iota_{[X_f,X_g]}\eta.
\end{equation}
Again, in Darboux coordinates the Jacobi bracket reads  
\begin{equation}
\{g,f\}_\eta = \left(g \frac{\partial f}{\partial s} - \frac{\partial g}{\partial s} f\right) + p_{\mu} \left(\frac{\partial g}{\partial s} \frac{\partial f}{\partial p_{\mu}} - \frac{\partial g}{\partial p_{\mu}} \frac{\partial f}{\partial s}\right) + 
\left(\frac{\partial g}{\partial q^{\mu}} \frac{\partial f }{\partial p_{\mu}} - \frac{\partial g}{\partial p_{\mu}} \frac{\partial f }{\partial q^{\mu}}\right). \label{eq:jacobibracketscoord}
\end{equation}
We refer the reader to \cite{Aroldmechanics,de2019contact,Bravetti2017,bravetti2019contact,gaset2019new} for further details.
For our scope, it will be important in the following to have an explicit expression for the Jacobi bracket of monomial functions, that is,
\begin{align}
\left\{\mu q^i  p^j s^r, \bar{\mu}  q^{\bar{i}}  p^{\bar{j}} s^{\bar{r}} \right\}_\eta &
=
\mu \bar{\mu}\left(\left[(1-j) \bar{r}+(\bar{j}-1) r \right]  q^{i+\bar{i}} p^{j+\bar{j}} s^{r+\bar{r}-1} \label{eq:monomial} 
+\left(i \bar{j} -\bar{i} j\right) q^{i+\bar{i}-1} p^{j+\bar{j}-1} s^{r+\bar{r}}\right),
\end{align}
where $\mu,\bar{\mu} \in \mathbb{R}$ and $i,j,r,\bar{i},\bar{j},\bar{r}\in \mathbb{N}.$

\subsection{A contact Hamiltonian formulation of Li\'enard Systems}
It is well known that any dynamical system on an $n$-dimensional manifold $Q$ of the form $\dot x^{i}=X^{i}(x)$ 
can be extended to a Hamiltonian system defined on the $2n$-dimensional phase--space $T^*Q$.
This can be achieved with the introduction of
the conjugate momenta $\tilde p_{i}$
in order to define the Hamiltonian
\begin{equation}\label{pXH-compact}
H(x,\tilde p)=\tilde p_{i}X^{i}(x)\,.
\end{equation} 
A direct
computation
shows that when we consider only the dynamics on the original $x$-variables, then we recover the original $n$-dimensional system.
For instance, in the case of Li\'enard systems~\eqref{eq:Li\'enard2d}, the Hamiltonian reads
\begin{equation}\label{pXH}
	H(x,y,\tilde p_{1},\tilde p_{2})=\tilde p_{1}y-\tilde p_{2}(g(x)+f(x)y),\quad (\tilde p_1, \tilde p_2) = (\tilde p_{x},\tilde p_{y})\,.
\end{equation}
In~\cite{choi1973extended}, such approach has been used to derive a Hamiltonisation of Li\'enard systems in such extended phase--space 
that was then shown to be useful to perform perturbation theory. 
Moreover, in~\cite{Pihajoki2014} a similar extension, but with a suitably defined new Hamiltonian that non-trivially couples the variables, has been used in order to develop geometric integrators in
the extended phase--space and then used e.g.~in the case of the van der Pol oscillator.

In principle one could use the Hamiltonian~\eqref{pXH} and perform a splitting in order to obtain new geometric integrators that are symplectic in the extended phase--space.
However, we see from the form of~\eqref{pXH} that it is linear in the momenta, meaning that
it is naturally associated with a contact Hamiltonian on the $(2n-1)$-dimensional projectivised cotangent bundle $PT^{*}Q$, endowed with the contact structure inherited from the canonical symplectic structure of $T^*Q$~\cite{arnol2013mathematical,blair2010riemannian}.
The procedure to perform such reduction is quite simple in this case and it is reviewed e.g.~in the recent work~\cite{vdsm2018}. 
In order to avoid clutter of notation, from now on we focus on the case $Q=\mathbb{R}^2$, which is the relevant case for our study:
we start with~\eqref{pXH} and consider a connected component of the open set in which $\tilde p_{2}\neq 0$. 
On such set we can define the coordinates $(q=x,s=y,p=-\frac{\tilde p_{1}}{\tilde p_{2}})$, which serve as Darboux coordinates on $PT^*\mathbb R^2$. 
Finally, we define the contact Hamiltonian 
\begin{equation}\label{cpXH}
\cH(q,p,s)=-\frac{1}{\tilde p_{2}}H(x,y,\tilde p_1,\tilde p_2)=pX^{1}(q,s)-X^{2}(q,s)\,.
\end{equation}
A direct calculation then shows that the restriction of the resulting contact Hamiltonian system to the $(q,s)$ plane recovers the original system.

By means of the above prescription, we arrive at the following result for Li\'enard systems.
\begin{theorem}[Hamiltonisation of Li\'enard systems]\label{th:contactLi\'enard}
Li\'enard systems are contact Hamiltonian systems, with Hamiltonian of the form
\begin{equation}
\cH = p s + f(q) s + g(q). \label{contactLi\'enard}
\end{equation}
The associated contact Hamiltonian system is
\begin{align}
\dot{q}&= s \label{dotqLi\'enard}\\
\dot{s}&= - f(q) s - g(q), \label{dotsLi\'enard}\\
\dot{p}&= - p^2 - f(q) p - f'(q) s -g'(q)\,. \label{dotpLi\'enard}
\end{align}
From the first two equations we recover the original Li\'enard system in the $(q,s)$-space, while the third equation is decoupled.
\end{theorem}

\begin{example}[The van der Pol oscillator revisited]
\label{ex:contactformulationvanderpol}
As we have already seen in Section \ref{sec:reviewLi\'enardsystems} the van der Pol equation is a particular case of a Li\'enard system, which is obtained by choosing $f(x)$ and $g(x)$ as
\begin{equation}
f(x)=-\epsilon (1-x^2), \qquad
g(x)=x\,.
\end{equation}
Consequently the contact Hamiltonian in this case reads
\begin{equation}
	\cH = p s -\epsilon (1-q^2) s +q\,, \label{eq:hamvanderpol}
\end{equation}
and the corresponding contact Hamiltonian systems is
\begin{equation}
	\begin{cases}
	\dot{q} = s\\
	\dot{s} = \epsilon(1-q^2) s - q\\
	\dot{p} = -1-p^2+\epsilon\left[(1-q^2)p -2qs\right]\,.
	\end{cases} \label{eq:contactdynsysvanderpol}
\end{equation}
As expected, from the first two equations we recover the original van der Pol equation~\eqref{vanderpol}.
\end{example}

\begin{remark}
    For $s\neq 0$ and setting the appropriate initial condition $p_0=-f(q_0)-g(q_0)/s_0$, $p(t)$ derived from \eqref{dotpLi\'enard} turns out to be the slope of the tangent $\frac{ds}{dq}$ to the orbit of the system at each point $(q(t), s(t))$ of its evolution.
    This stems from the fact that \eqref{dotqLi\'enard}-\eqref{dotpLi\'enard} are the characteristic equations of the Hamilton-Jacobi equation for \eqref{contactLi\'enard}.
    Details of this derivation are in preparation by \cite{liu:contactext}.
\end{remark}

\begin{remark}
	The reduction procedure that led us to \eqref{cpXH} is not unique.
	Indeed, we could have selected the connected component in which $\tilde p_{1}\neq 0$ and set $(q=y,s=x,p=-\frac{\tilde p_{2}}{\tilde p_{1}})$. The corresponding contact Hamiltonian for Liénard systems is:
	\begin{align}
		\mathcal{K}(q,p,s)
		&= -\frac{1}{\tilde p_{1}}H(x,y,\tilde p_1,\tilde p_2)=pX^{2}(q,s)-X^{1}(q,s)\\
		&= -p(f(s)q+g(s)) -q\,.
	\end{align}
	Beware that in this case $X^1(q,s)=q$ and $X^2(q,s)=-f(s)q-g(s)$, that is, the roles of $q$ and $s$ are switched, and the resulting system is \begin{equation}
	\begin{cases}
	\dot{q} = -f(s)q-g(s)\\
	\dot{s} = q\\
	\dot{p} = 1+pf(s)+p\left(pqf'(s)+g'(s)\right)\,,
	\end{cases} \label{eq:contactvdp2}
	\end{equation}
	which is equivalent to~\eqref{dotqLi\'enard}-\eqref{dotpLi\'enard} for the $(q,s)$ part, but not so much for $p$.
	
	The choice of reduction, in the case at hand, was dictated by numerical convenience: the Hamiltonian $\cH$ from \eqref{cpXH} resulted in a simpler form of the algorithm providing better results.
\end{remark}

\section{Geometric numerical integration of Li\'enard systems}\label{sec:GeomInt}

\subsection{Contact splitting integrators}
\label{sec:contactnumericalintegrators}

Contact splitting integrators are a class of geometric integrators recently introduced in the context of celestial mechanics \cite{Bravetti2020}. 
They are the contact analogues of the well-known symplectic splitting integrators. 

Let $\cH$ be a contact Hamiltonian which is separable into a sum of functions
\begin{equation}
	\cH(q^i,p_i,s) = \sum_{j=1}^{N} \ch_j (q^i,p_i,s).
\end{equation}
Then, the Hamiltonian vector field associated with $\cH$ is separable as well
\begin{equation}
	X_{\cH} = \sum_{j=1}^{N} X_{\ch_j}.
\end{equation}
If moreover, each of the $X_{\ch_{j}}$ is \emph{exactly integrable}, meaning that there exists a closed-form solution for its flow,

then we can approximate the dynamics of $X_{\cH}$ to second order in $\tau$ with contact maps according to the following proposition.
\begin{Prop}[Contact splitting integrators]
	In the hypotheses above, let
	$e^{tX_{\ch_{j}}}$ denote the map given by time-$t$ exact flow of each vector field $X_{\ch_{j}}$, for $j=1,\dots,N$.
	Then
	\begin{equation}
  		S_2(\tau) = e^{\frac{\tau}{2} X_{\ch_1}} e^{\frac{\tau}{2} X_{\ch_2}} \cdots e^{\tau X_{\ch_N}} \cdots e^{\frac{\tau}{2} X_{\ch_2}}e^{\frac{\tau}{2} X_{\ch_1}} \label{eq:secondorderintegrator}
	\end{equation}
	is a second order contact numerical integrator, meaning that each map is a contactomorphism.
\end{Prop}

From knowledge of the second order contact integrator~\eqref{eq:secondorderintegrator} and using Yoshida's standard formulation for the composition~\cite{Yoshida}, 
we can construct two types of contact integrators of any even order;
the difference between the two methods is that one involves exact coefficients for the calculation of the new time step, while
the other uses approximated coefficients and involves a smaller number of map computations per iteration. The two methods are summarised in the following propositions.

\begin{Prop}[Integrator with exact coefficients]\label{prop:exact}
    If $S_{2n}(\tau)$ is an integrator of order $2n$, then the map
    \begin{equation}
        S_{2n+2} (\tau) = S_{2n}(z_1\tau) S_{2n}(z_0\tau) S_{2n}(z_1\tau), \label{2n+2orderexact}
    \end{equation}
    with $z_0$ and $z_1$ given by
    \begin{equation}%
	\label{z0z1}
        z_0(n)=-\frac{2^{\frac{1}{2n+1}}}{2-2^{\frac{1}{2n+1}}}, \hspace{0.5cm} z_1(n)=\frac{1}{2-2^{\frac{1}{2n+1}}};
    \end{equation}
    is an integrator of order $2n+2$.
\end{Prop}
\begin{Prop}[Integrator with approximated coefficients]\label{prop:approx}
    There exist $m\in\mathbb{N}$ and a set of real coefficients $\{w_j\}^m_{j=0}$ such that the map
    \begin{equation}
        S^{(m)}(\tau)=S_2(w_m\tau) S_2(w_{m-1}\tau) \cdots S_2(w_{0}\tau) \cdots S_2(w_{m-1}\tau) S_2(w_m\tau), \label{msymm}
    \end{equation}
    is an integrator of order $2n$.
\end{Prop}

In Table~\ref{numsol} we list the values of the approximated coefficients $\{w_j\}^m_{j=0}$ for three different $6$th order integrators, labelled as A, B and C. 
Note that $w_0 := 1 - 2\sum_{j=1}^m w_i$.
\begin{table}[ht!]
    \centering
    \begin{tabular}{cccc}
        \toprule
        & A & B & C \\
        \midrule
        $w_0$ & $1.315186320683906$& $2.37635274430774$
        & $2.3894477832436816$
        \\
        \midrule
        $w_1$ & $-1.17767998417887$ & $-2.13228522200144$ & $0.00152886228424922$ \\
        $w_2$ & $0.235573213359357$ & $0.00426068187079180$ & $-2.14403531630539$  \\
        $w_3$ & $0.784513610477560$ & $1.43984816797678$ & $1.44778256239930$ \\
        \bottomrule
    \end{tabular}
        \caption{The coefficients $w_{i}$ for three different 6th order integrators.}%
	\label{numsol}
\end{table}

\begin{remark}\label{rmk:familyA}
	The splitting integrator with approximate coefficients labeled as A is the better performer among the approximate splitting integrators of 6th order presented here.
	This can be related to the fact that its largest coefficient is the smallest among the approximate integrators.
\end{remark}

\subsection{Modified Hamiltonian and error analysis}\label{subsec:modifiedH}

One of the main advantages of using contact splitting integrators is the possibility to have a direct error control by using the modified equations
obtained from the modified Hamiltonian that results from 
the Baker-Campbell-Hausdorff (BCH) formula
(see~\cite{Bravetti2020} for further details on the derivation of the modified Hamiltonian in the contact case).
Indeed, for an integrator of order $2n$ multiple applications of the BCH formula give~\cite{Bravetti2020}
\begin{equation}\label{eq:S2taubis}
S_{2n}(\tau) = \exp\left\{\tau X_\cH + \sum_{i=n}^{\infty} \tau^{2i+1} X_{2i+1} \right\},
\end{equation}
where all the corrections $X_{2i+1}$ are Hamiltonian vector fields.
Therefore $S_{2n}(\tau)$ is the time-$\tau$ flow of a Hamiltonian vector field, and its associated Hamiltonian, called \emph{the modified Hamiltonian}, 
can be written formally as the power series
\begin{equation}\label{eq:modH2}
\cH_{mod,2n}(q^a , p_a,s;\tau) = \cH (q^a , p_a,s) + \sum_{i=n}^{\infty} \tau^{2i} \Delta \cH_{2i} (q^a , p_a,s)\,,
\end{equation}
where the subscript $2n$ in $\cH_{mod,2n}$ denotes the fact that it is associated with an integrator of order $2n$, and
$\Delta \cH_{2i}$ are the Hamiltonian functions associated with the Hamiltonian vector fields $X_{2i+1}$, that is,
\begin{equation}
	-\Delta \cH_{2i} (q^a , p_a,s) = \iota_{X_{2i+1}} \eta.
\end{equation}

Plugging~\eqref{eq:modH2} into the contact Hamiltonian equations that stem from~\eqref{eq:hamvecfield}, we obtain \emph{the modified equations},
which are the equations whose time-$\tau$ flow gives exactly the integrator $S_{2n}(\tau)$. Therefore studying the modified equations and their relation 
with the original
equations gives us important information on the modifications introduced by the integrator on the original system.

\subsection{Geometric numerical integration of Li\'enard systems}
The application of the contact splitting integrators introduced in Section~\ref{sec:contactnumericalintegrators} 
to Li\'enard systems starts with the splitting of the contact Hamiltonian~\eqref{contactLi\'enard} as
	\begin{equation}
	\cH =  \underbrace{p s}_{C} + \underbrace{f(q) s}_{A} + \underbrace{g(q)}_{B}, \label{eq:hamsplitting}
	\end{equation}
and the consequent identification of the corresponding vector fields
\begin{align}
	X_A&= - \bigg(p f(q) + s f'(q)\bigg)\frac{\partial}{\partial p} -s f(q) \vecs,\label{eq:Li\'enardXA} \\ 
	X_B&= -g'(q) \frac{\partial}{\partial p} - g(q) \vecs,\label{eq:Li\'enardXB} \\
	X_C&= s \frac{\partial}{\partial q} - p^2 \frac{\partial}{\partial p}.\label{eq:Li\'enardXC}
\end{align}
The structure of this splitting ensures the exact integrability condition for any choice of the functions $f(q)$ and $g(q)$. Indeed, the time-$\tau$ flow maps are
explicitly given by
\begin{align}
    e^{\tau X_A} &\longrightarrow
	\begin{cases}
	q_{i+1} = q_i\\
	p_{i+1} = e^{-\tau f(q_i)} (p_i - f'(q_i) s_i \tau \epsilon )\\
	s_{i+1} = e^{-\tau f(q_i)} s_i
	\end{cases} \notag
	\\e^{\tau X_B} &\longrightarrow \begin{cases}
	q_{i+1} = q_i\\
	p_{i+1} = -g'(q_i)\tau + p_i\\
	s_{i+1} = -g(q_i) \tau + s_i
	\end{cases}
	 \label{eq:lienmaps}
	\\
	e^{\tau X_C} &\longrightarrow
	\begin{cases}
	q_{i+1} = q_i + s_i \tau\\
	p_{i+1} = \frac{p_i}{1+p_i \tau}\\
	s_{i+1} = s_i
	\end{cases} \notag
\end{align}

\begin{example}[The van der Pol oscillator yet again]
	\label{example:vdp}
	Applying the above splitting to the Hamiltonian~\eqref{eq:hamvanderpol} we obtain
	\begin{equation}
	\cH = \underbrace{p s}_C \underbrace{{-}\,\epsilon (1-q^2) s}_A \underbrace{{+}\,q}_B\,, \label{eq:vdphamsplittedparts}
	\end{equation}
	and the corresponding time-$\tau$ flow maps are
	\begin{align}
	e^{\tau X_A} &\longrightarrow
	\begin{cases}
	q_{i+1} = q_i\\
	p_{i+1} = e^{\left(q_i^2-1\right) \tau \epsilon } (p_i - 2 q_i s_i \tau \epsilon )\\
	s_{i+1} = e^{-\tau \epsilon (1 - q_i^2)} s_i
	\end{cases}\notag
	\\
	e^{\tau X_B} &\longrightarrow \begin{cases}
	q_{i+1} = q_i\\
	p_{i+1} = p_i-\tau\\
	s_{i+1} = s_i-\tau q_i
	\end{cases}  \label{eq:vdpmaps}
	\\
	e^{\tau X_C} &\longrightarrow
	\begin{cases}
	q_{i+1} = q_i + s_i \tau\\
	p_{i+1} = \frac{p_i}{1+p_i \tau}\\
	s_{i+1} = s_i
	\end{cases} \notag
	\end{align}
\end{example}

In the next section we present the numerical and analytical results
of the application of various splitting integrators based on the maps~\eqref{eq:vdpmaps} to the van der Pol oscillator.
To fix the notation, when referring to a particular splitting, 
we will write e.g.~$S_2(\tau)(CBABC)$ to indicate that we are using the 2nd order integrator obtained using the splitting~\eqref{eq:secondorderintegrator} of the maps~\eqref{eq:vdpmaps} composed in the order indicated in parentheses.

\section{Geometric numerical integration of the van der Pol oscillator: numerical vs analytical results}
\label{sec:numericalandanalytical}

\subsection{Numerical results}\label{sec:numericalresults}

We split the analysis into three different cases, labelled by the value of the nonlinear coupling parameter~$\epsilon$: for $\epsilon=0$ we recover the harmonic oscillator on the plane $(q,s)$;
for $\epsilon\ll1$ and  $\epsilon\sim1$ we are in the non-stiff regime; for $\epsilon\gg1$ we are in the stiff regime.

It is well-known that to approximate the limit cycle with Euler-type methods, one cannot choose the time step $\tau$ independently of $\epsilon$, even in the non-stiff case $\epsilon \ll 1$ \cite{ChRaSt2018}: for example the Euler method requires $\tau \ll \epsilon$ and the exponential midpoint method requires $\tau^3 \ll \epsilon$.

In the rest of this section we will focus on the performance of our algorithm in the preservation of the limit cycle.
As we will see, our methods accurately preserve the limit cycle of the van der Pol oscillator when $\tau \ll  1$: this allows for much larger step sizes than Euler-type methods when integrating Li\'enard systems. 

\subsubsection{$\epsilon = 0$ (harmonic oscillator)}\label{sec:numericalHO}
Figure~\ref{fig:orbitse=0} shows the solutions in the $(q,s)$-plane for different time steps $\tau$ and with the same initial condition $(q_0,p_0,s_0) = (2,0,0)$.
\begin{figure}[ht!]
	\centering
	\includegraphics[width=0.8\linewidth]{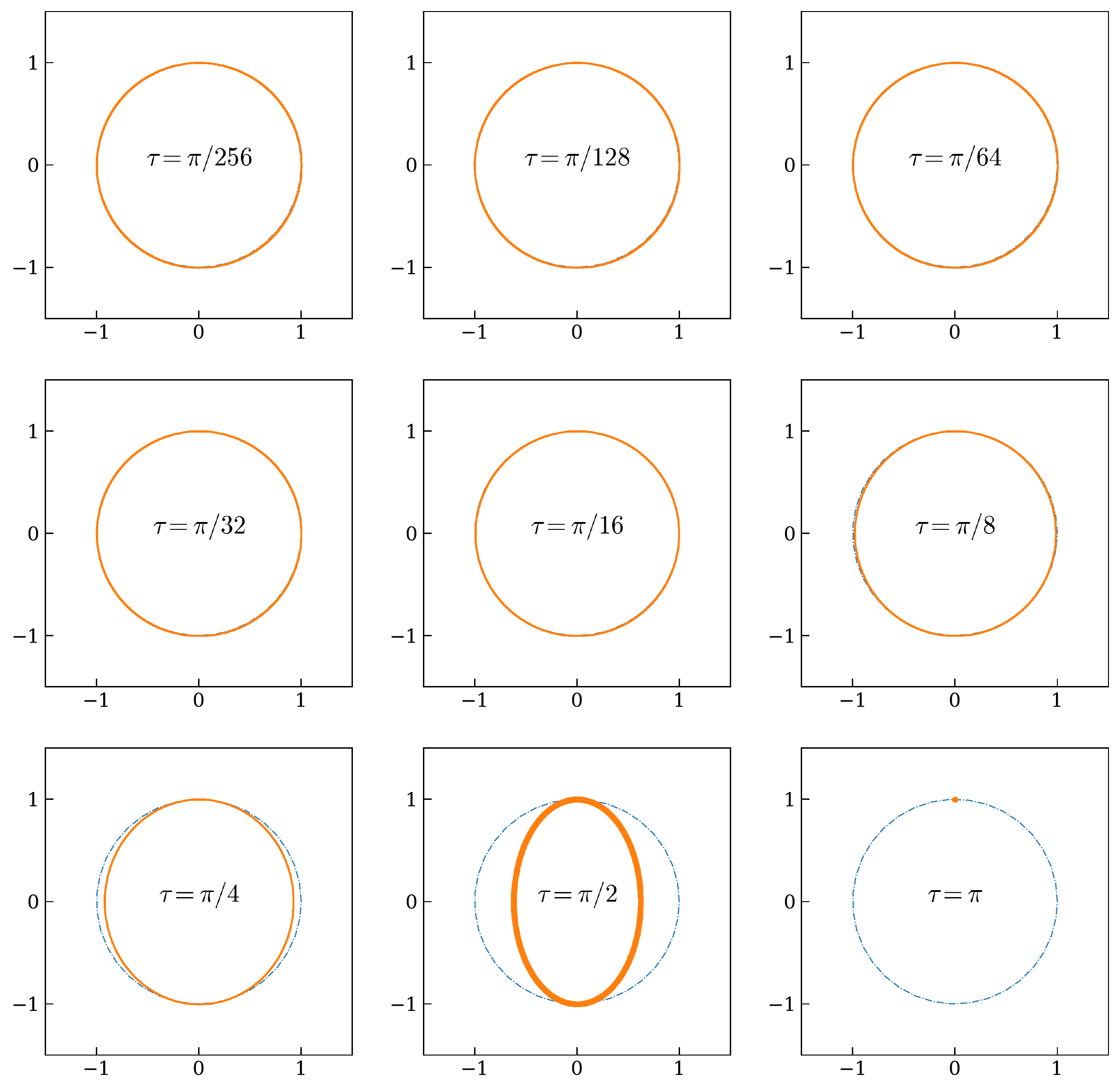}
	\caption{Orbit of the van der Pol oscillator with $\epsilon=0$ (harmonic oscillator) with initial condition $(q_0,p_0,s_0) = (0,0,1)$ integrated for different values of the time step $\tau$. The dashed blue line shows the exact solution.}
	\label{fig:orbitse=0}
\end{figure}
We can observe that the integrator is stable at least until
the surprisingly large value $\tau \sim \pi/2 > 1$.
By increasing the time step the typical circular orbit of the harmonic oscillator becomes more elliptic, and the period changes.
In Figure~\ref{fig:periodvstaue=0} we plot the relation between the time step and the period of the orbits obtained from numerical simulations.
Even though the frequency changes, we can see that the variation remains well under control for all values of $\tau\in(0,1]$.

\begin{figure}[ht!]
	\centering
	\includegraphics[width=0.8\linewidth]{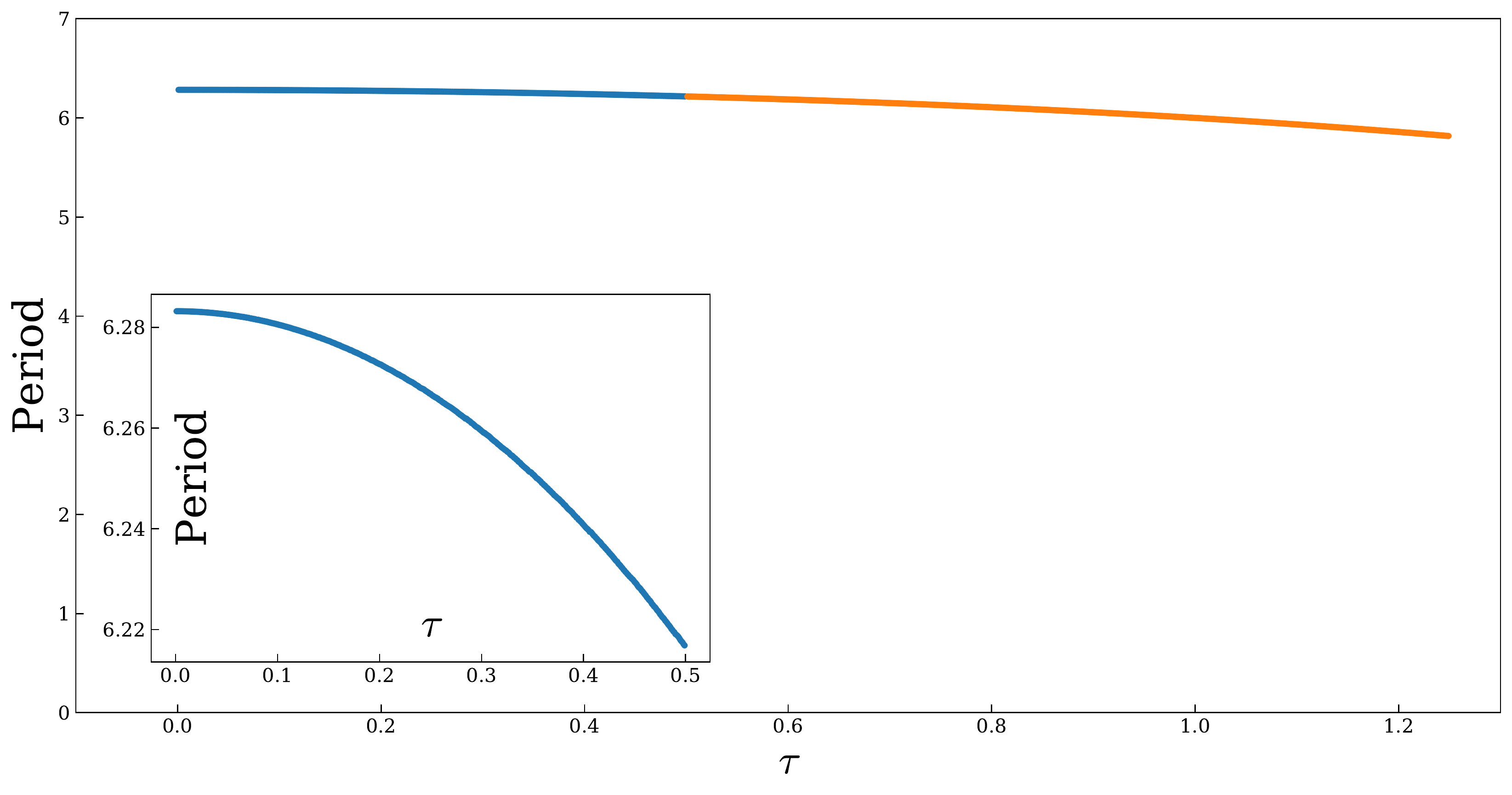}
	\caption{van der Pol oscillator with $\epsilon=0$ (harmonic oscillator). Dependence of the 
	period of the numerical solution with respect to the time step. The inset plot is a closeup of the periods for $\tau\in[0.001,0.5]$}
	\label{fig:periodvstaue=0}
\end{figure}

\subsubsection{$\epsilon\ll1$ and $\epsilon \sim 1$ (non-stiff regime)}
\label{subsubsec:nonstiffcasenumerical}
In Figure~\ref{fig:orbitsesim1} we show the persistence of the limit cycle for different values of $\epsilon\in\{0.1,0.5,2,4\}$ and $\tau\in [\pi/256,\pi/2]$.
\begin{figure}[ht!]
	\centering
	\includegraphics[width=0.8\linewidth]{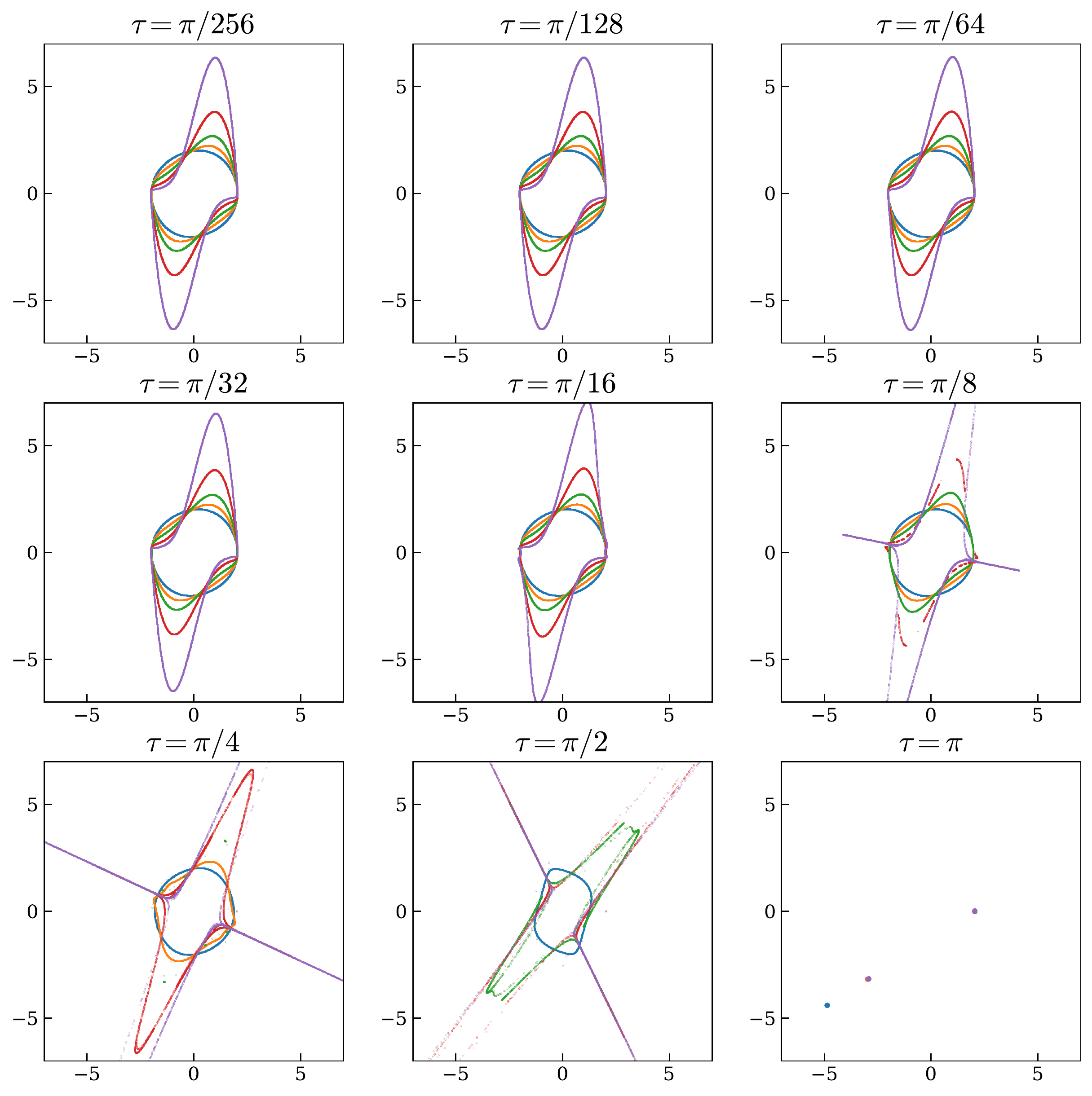}
	\caption{Limit cycle of the van der Pol oscillator for values of $\epsilon=$ 0.1 (blue), 0.5 (orange), 1 (green), 2 (red), 4 (purple) and with different time steps.}
	\label{fig:orbitsesim1}
\end{figure}
Clearly the limit cycle is preserved also for very large values of $\epsilon$ and $\tau$ in this range.
Moreover, the very long integration time, with $t \in [0,10000]$, is an evidence of the stability of the integrator.
Finally, the dependence of the period and the frequency of the limit cycle with respect to the time step shown in Fig.~\ref{fig:periodfreq} 
is very similar to that of the harmonic oscillator.
\begin{figure}[ht!]
	\centering
	\includegraphics[width=\linewidth]{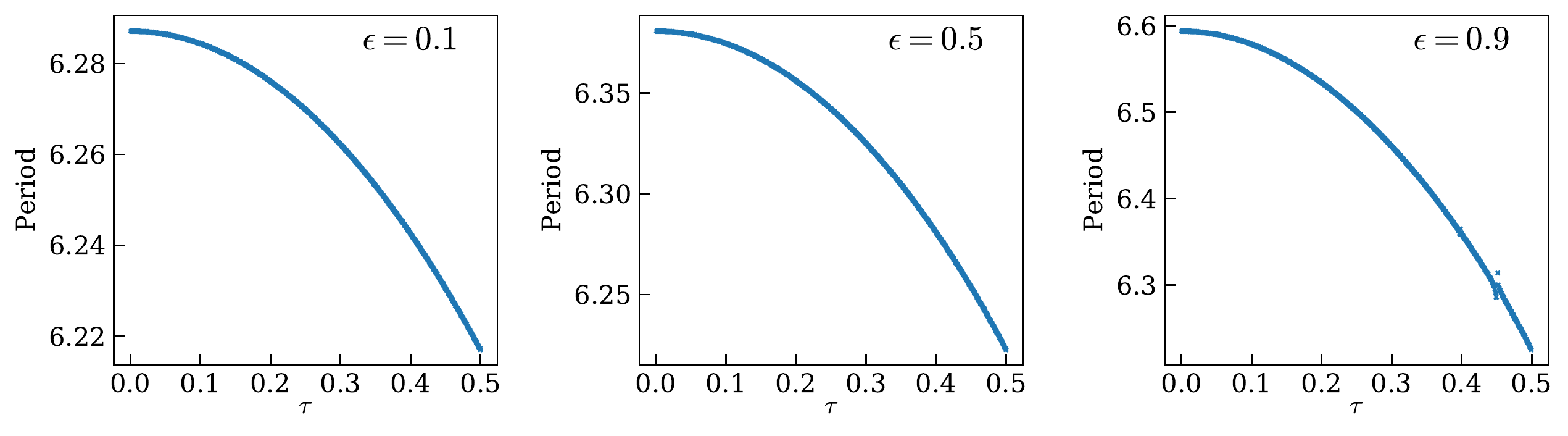}
	\caption{Dependence of the period of the numerical solution of the van der Pol limit cycle with respect to the time step for $\epsilon\in\{0.1, 0.5, 0.9\}$ increasing from left to right.}
	\label{fig:periodfreq}
\end{figure}

\subsubsection{$\epsilon\gg1$ (stiff regime)}
\label{subsubsec:stiffcasenumerical}

To better understand what happens in the stiff case $\epsilon\gg1$, it is convenient to perform, after the integration, the so-called \emph{Li\'enard transformation}~\cite{ChRaSt2018, lienard1928}
\begin{equation}
	\begin{cases}
	\underbar q = q \\
	\underbar s = q - \frac{q^3}{3} - \frac{s}{\epsilon}.
	\end{cases}\label{eq:Li\'enardtransform}
\end{equation}
This change of variables transforms the dynamics into
\begin{equation}
	\begin{cases}
	\dot {\underbar q} = \epsilon\left( \underbar q - \frac{{\underbar q}^3}{3} - \underbar s\right) \\
	\dot {\underbar s} = - \underbar q / \epsilon.
	\end{cases}
\end{equation}
and enables a nice geometric description of the limit cycle.
Indeed, the $\underbar q$ nullcline, which is the locus of points such that $\dot {\underbar q} = 0$, is given by the cubic $\underbar s = \underbar q - \frac{{\underbar q}^3}3$.
Since $\underbar q$ evolves much faster than $\underbar s$, the solutions are quickly attracted by the cubic nullcline.
Once there, they move slowly along the curve until they reach an extremum, at which point they quickly jump horizontally to the other branch of the nullcline. This periodic motion that jumps back and forth on the nullcline is the attractive limit cycle of the stiff van der Pol oscillator.

\begin{figure}[ht!]
	\centering
	\includegraphics[width=0.8\linewidth]{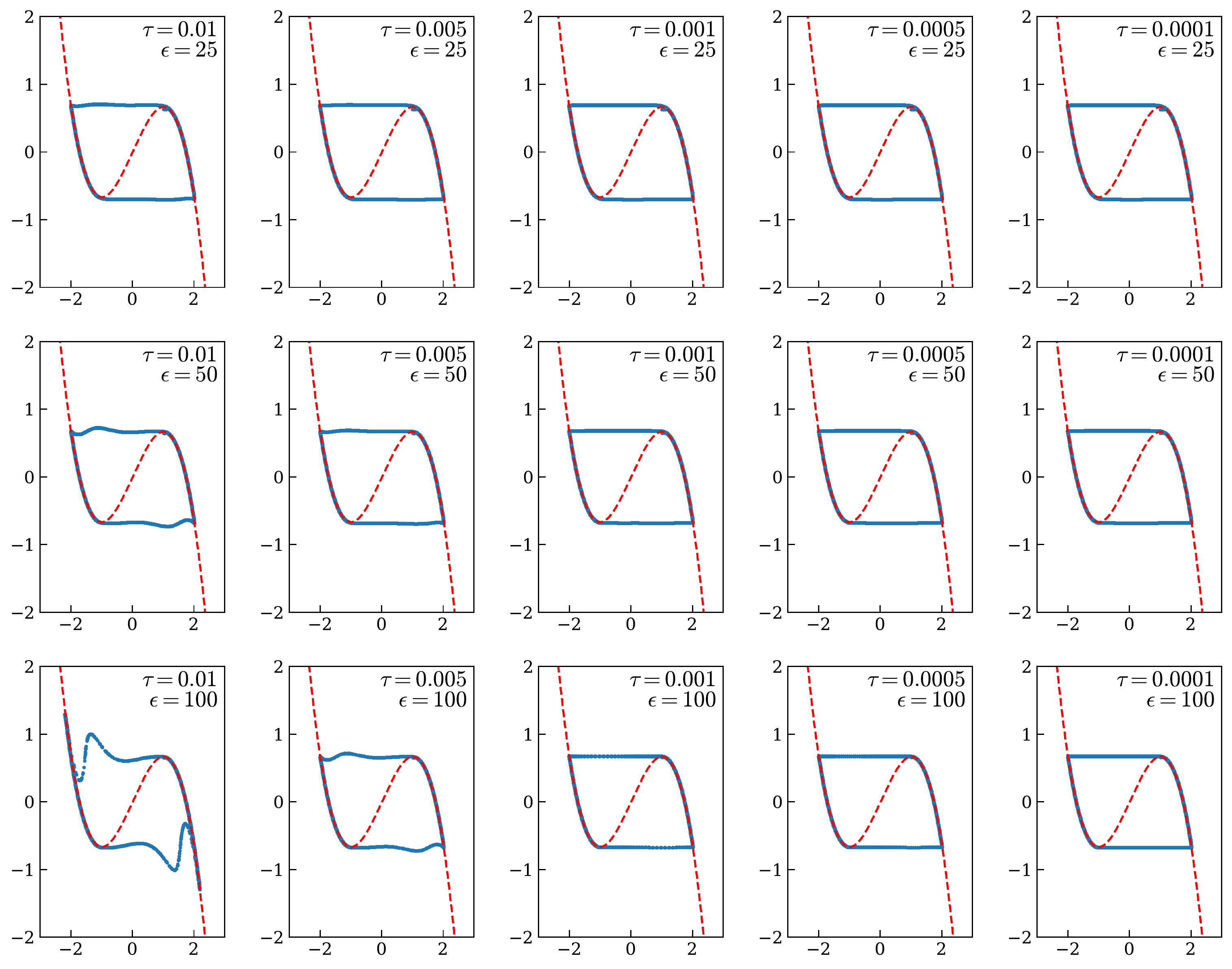}
	\caption{Orbits for the stiff van der Pol oscillator for different values of the coupling $\epsilon$ and of the time step~$\tau$ after the Li\'enard transformation: with $\epsilon\in\{25, 50, 100\}$ increasing from top to bottom  and ${\tau\in\{10^{-2}, 5\times10^{-3}, 10^{-3}, 5\times10^{-4}, 10^{-4}\}}$ decreasing from left to right.}
	\label{fig:bige}
\end{figure}

Figure~\ref{fig:bige} shows the cubic nullcline and the numerically simulated attractor for $\epsilon\in\{25,50,100\}$ and for different values of the time step.
As one can observe, the limit cycle is preserved also for large values of the nonlinear coupling, although it suffers from a distortion for larger values of $\tau$: this is especially clear in the first picture of the last row of plots of Figure~\ref{fig:bige}, corresponding to $\epsilon=100$ and $\tau=0.01$.

\subsection{Analytical results}
\label{sec:analiticalstudy}
In this section we provide an analytical study of the contact splitting integrators for the van der Pol oscillator based on the modified equations.
We start by providing two general properties of the modified equations that are of special importance.
 
As we have seen in Example~\ref{ex:contactformulationvanderpol}, in the contact formulation of the van der Pol oscillator the equations 
for ${q}$ and ${s}$ are independent of $p$, as it should be. 
Clearly, given that the maps for $q$ and $s$ in~\eqref{eq:vdpmaps} are all independent of $p$, any splitting integrator will satisfy this property too. However, it is instructive to recover this result by using the modified Hamiltonian, since in the proof we will find out an important property of $\cH_{mod}$, i.e.~that it is linear in $p$, as it is the original Hamiltonian~\eqref{eq:hamvanderpol}.
This is the content of the next result.

\begin{Prop}
	\label{prop:singular}
	For any contact splitting integrator,
	the corresponding modified Hamiltonian  $\cH_{mod}$ is linear in $p$. 
	It follows that the modified equations for $q$ and $s$ are independent of $p$.
\end{Prop}
\begin{proof}
We prove first the second part: the claim is that if $\cH_{mod}$ is linear in $p$, then the corresponding modified equations for $q$ and $s$ do
not depend on $p$. By a direct look at the general contact Hamiltonian equations~\eqref{eq:hamvecfield}, this is clearly true.
Now let us prove that $\cH_{mod}$ is indeed linear in $p$: considering the splitting in~\eqref{eq:vdphamsplittedparts}, we have that $A= - \epsilon \left( 1-q^2\right) s$, $B= q$, and $C= ps$, are all polynomials in $q,p,s$ and that only $C$ depends (linearly) on $p$. 
Therefore, we see from~\eqref{eq:monomial} that by commuting $A$, $B$ and $C$ we can only obtain terms that are at most linear $p$.
Then again, by commuting two terms that are at most linear in $p$, we see from~\eqref{eq:monomial} that we always obtain terms that are at most linear in $p$. 
We conclude that the modified Hamiltonian is at most linear in $p$.  
We conclude that $\cH_{mod}$ is indeed linear, because otherwise in the modified equations we would have $\dot q=0$, which is clearly not the case.
\end{proof}

Furthermore, we observe that when the time step $\tau \neq 0$  any truncation of the modified equations is likely to possess new spurious equilibria. 
This is so since at any order the corresponding vector fields are polynomials in $q,p,s$ of increasing order. 
Therefore it is important to actually prove that $(q,s)=(0,0)$ is the only fixed point (considering only the dynamics projected to the $(q,s)$ plane) for the integrator and that it is unstable, as we show in the next result.

\begin{Prop}\label{prop:originequilibrium}
	Restricted to the plane $(q,s)$, the integrator $S_2(\tau)(CBABC)$ has a unique fixed point at $(0,0)$ which is unstable.
	Furthermore, both the eigenvalues $\lambda_{1,2}$ of the Jacobian of the mapping $(q_i, s_i) \mapsto (q_{i+1}, s_{i+1})$ satisfy $|\lambda_{1,2}| > 1$ for all $\tau > 0$ and $\epsilon > 0$.
\end{Prop}

\begin{proof}
The proof is based on writing explicitly the action of the integrator on an initial condition, that is, 
	we apply $e^{\tau/2 X_C}e^{\tau/2 X_B}e^{\tau X_A}e^{\tau/2 X_B}e^{\tau/2 X_C}$ to $(q_{i}, s_{i})$, to obtain
	\begin{align}
	\begin{cases}
	q_{i+1} = q_i + \frac{\tau}{2}s_{i}+\frac{\tau}{2}s_{i+1}\\
%	p_{i+1} = \frac{p_i}{1+p_i \tau}\\
	s_{i+1} = e^{\epsilon\tau(1-(q_i +\frac{\tau}{2} s_i)^{2})}\left[s_{i}-\frac{\tau}{2}\left(q_{i}+s_{i}\frac{\tau}{2}\right)\right]-\frac{\tau}{2}\left(q_{i}+s_{i}\frac{\tau}{2}\right)\,.
	\end{cases} \label{eq:integratormap}
	\end{align}
Now when we impose the condition for a fixed point
	\begin{align}
	\begin{cases}
	q_{i+1} = q_i \\
%	p_{i+1} = \frac{p_i}{1+p_i \tau}\\
	s_{i+1} = s_{i}\,.
	\end{cases} \label{eq:integratormap2}
	\end{align}
using the second equation in~\eqref{eq:integratormap2} into the first equation in~\eqref{eq:integratormap} we obtain
	$$
	q_{i+1}=q_{i}+\tau s_{i}=q_{i}\,,
	$$
which is true if and only if $s_{i}=0$.

Next, we substitute $s_{i}=0=s_{i+1}$ into the second equation in~\eqref{eq:integratormap} and we obtain
	$$
	0=s_{i+1}
	%= e^{\epsilon\tau(1-(q_i )^{2})}\left[-\frac{\tau}{2}\left(q_{i}\right)\right]-\frac{\tau}{2}\left(q_{i}\right)
	=-\frac{\tau}{2}q_{i}\left[e^{\epsilon\tau(1-(q_i )^{2})}+1\right]\,,
	$$
which is true if and only if 
$q_{i}=0$.

To prove that $(0,0)$ is unstable, we compute the
Jacobian of the map~\eqref{eq:integratormap} 
at $(0,0)$, and in particular we obtain that its determinant is $e^{\tau\epsilon}>1$, indicating that at least one eigenvalue has absolute value $>1$, which proves the instability.

To conclude the proof, let $\epsilon > 0$ and $\tau > 0$.
A direct computation shows that the eigenvalues of the Jacobian of the map~\eqref{eq:integratormap} at $(0,0)$ are
\begin{equation}
	\lambda_{1,2} = \frac14\left[\alpha \pm \sqrt{\beta}\right],
\qquad
	\alpha := (2-\tau^2) (e^{\epsilon \tau} + 1),
\qquad
	\beta := \alpha^2 - 16 e^{\epsilon \tau}.
\end{equation}
Depending on the sign of $\beta$ we have two cases: the eigenvalues are both real or they are complex conjugates.
\begin{description}
	\item[Case I) $\lambda_{1,2} \in \mathbb{C}$:]
	the eigenvalues are complex conjugates, therefore $|\lambda_{1}| = |\lambda_2|$.
	Since $\det J = \lambda_1\lambda_2 = e^{\epsilon \tau}$, we have $|\lambda_{1}| = |\lambda_2| = e^{\frac{\epsilon \tau}2} > 1$.
	\item[Case II) $\lambda_{1,2} \in \mathbb{R}$:]
	this happens when $\beta \geq 0$, that is
	\begin{equation}\label{eq:conditioncase2}
		\alpha \geq 4 e^{\frac{\epsilon \tau}2}.
	\end{equation}
	The fact that $\lambda_1 > 1$ follows from
	$
		\lambda_1 = \frac14\left[\alpha + \sqrt{\beta}\right]
		\geq \frac14 \alpha 
		\overset{\eqref{eq:conditioncase2}}{\geq}
		    e^{\frac{\epsilon \tau}2} > 1
	$.\\
	Let us now focus on $\lambda_2$. Notice that since $\lambda_1\lambda_2>0$ and $\lambda_1>0$, we necessarily have that $\lambda_2>0$. Therefore, it suffices to prove that
	\begin{equation}\label{eq:lambda2}
		\lambda_2 = \frac14\left[\alpha - \sqrt{\beta}\right] > 1.
	\end{equation}
	By repeatedly rearranging \eqref{eq:lambda2} and observing that \eqref{eq:conditioncase2} implies $\alpha-4>0$, we obtain that \eqref{eq:lambda2} is equivalent to the following inequalities
	\begin{align}
		\alpha-\sqrt{\beta}    &> 4\\
		% \alpha-4         &> \sqrt{\beta} \\
		% \alpha-4         &> \beta\\
		(\alpha-4)^2       &> \beta \\
		\alpha^2 - 8\alpha + 16 &> \alpha^2 - 16 e^{\epsilon \tau}\\
		\left[16 - 8 (2-\tau^2)\right] (e^{\epsilon \tau} + 1) &> 0.\label{eq:final}
	\end{align}
	Since $(2-\tau^2) < 2$, \eqref{eq:final} is always true, proving \eqref{eq:lambda2}.
\end{description}
\end{proof}

In what follows we split the analysis into three different cases depending on the value of $\epsilon$, as we did in the Section~\ref{sec:numericalresults}.

\subsubsection{$\epsilon = 0$ (harmonic oscillator)}
\label{subsec:exactHarmonic}
In this case we have a harmonic oscillator, for which each nontrivial trajectory has period $T=2\pi$.
Moreover, the maps \eqref{eq:vdpmaps} in this particular case are simplified (for instance, the map $e^{\tau X_A}$ becomes the identity)
and the modified Hamiltonian takes the remarkably simple expression
	\begin{equation}
	    \cH_{mod,2} = p s F(\tau) + q \,G(\tau) \label{eq:modhame=0}
	\end{equation}
where 
	\begin{align}
	F(\tau) &= 1-\frac{\tau ^2}{12}-\frac{\tau ^4}{120}-\frac{\tau ^6}{840}-\frac{\tau ^8}{5040}+\cO(\tau^{10}),\\
	G(\tau) &= 1+\frac{\tau ^2}{6}+\frac{\tau ^4}{30}+\frac{\tau ^6}{140}+\frac{\tau ^8}{630}+\cO(\tau^{10}).
	\end{align}
The corresponding modified system is thus
	\begin{equation}
	    \begin{cases}
        	\dot{q}(t) = s F(\tau) \\
        	\dot{s}(t) = -q G(\tau)\\
	\dot{p}(t) = -p^2 F(\tau )-G(\tau) 
	    \end{cases}
	\end{equation}
which is again exactly solvable (recall that $\tau$ is fixed), and the solution in $q$ and $s$ is a harmonic oscillator with frequency
	\begin{align}
	\omega(\tau) &= \sqrt{F(\tau) G(\tau)}
	=1 +\frac{\tau ^2}{24} + \frac{3 \tau ^4}{640} +\frac{5 \tau ^6}{7168}+ \frac{35 \tau ^8}{294912}+ \cO(\tau^{10})\,.
	\label{eq:frequency}
	\end{align}
In Fig.~\ref{fig:periodvstaue=0exact} we compare~\eqref{eq:frequency}
with the numerical results for the period and the frequency obtained in Section~\ref{sec:numericalHO}.
We observe that there is a very good agreement between the analytical expression up to the 8th order in $\tau$ and the numerical results.
	\begin{figure}[ht!]
		\centering
		\includegraphics[width=0.8\linewidth]{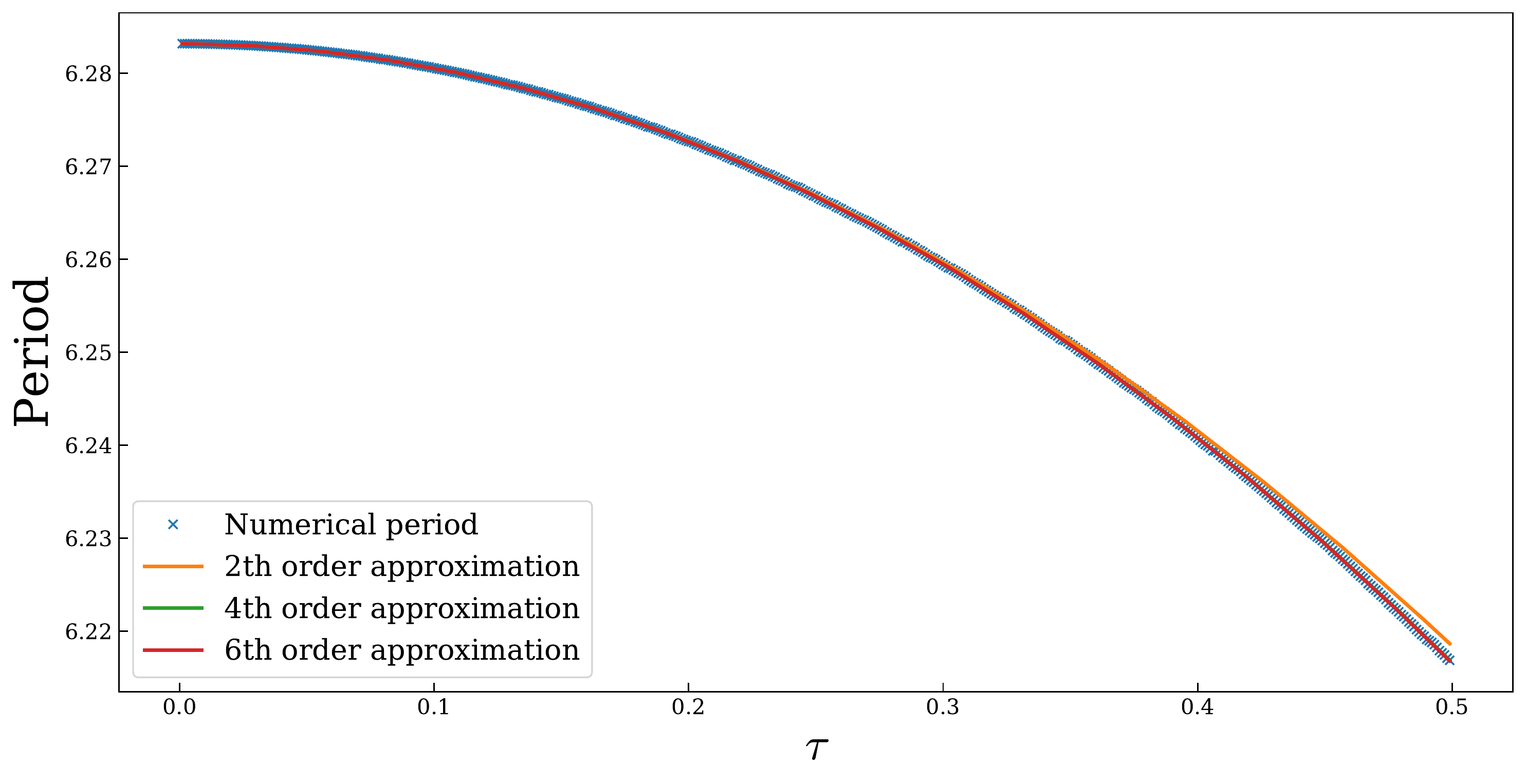}
		\caption{Dependence of the period of the numerical solution for the harmonic oscillator ($\epsilon=0$) with respect to the time step. The numerically estimated period is compared with the period computed from the modified equations.}
		\label{fig:periodvstaue=0exact}
	\end{figure}

\subsubsection{$\epsilon\ll1$ (non-stiff regime)}\label{sec:nonstiffregime2}
This regime can be studied using perturbation theory and therefore there are many results (see e.g.~\cite{Amore2018,Andersen1982}).
We study 
the persistence of the limit cycle for the contact splitting integrators in a way similar to~\cite{ChRaSt2018}, 
that means, we use the modified equations in order to provide some estimations on the amplitude and period of the limit cycle.
\begin{Prop}
	For any contact splitting integrator of order $2n$ based on the maps~\eqref{eq:vdpmaps},
	the projection of the numerical solutions of the van der Pol system~\eqref{eq:contactdynsysvanderpol} onto the $(q,s)$-plane have a 
	limit cycle at the approximate radius $r=2+\cO(\tau^{2n})$. 
	Moreover, the approximate radius of the $S_{2}(\tau)(CBABC)$ integrator, up to order 4  in $\tau$, is
	\begin{equation}
	    r=2-\frac{\tau ^2}{4}+\cO\left(\tau ^4\right).
	\end{equation}
	\label{prop:ray}
\end{Prop}
\begin{proof}

Let us consider a contact splitting integrator $S_{2n}(\tau)$ of order $2n$; using the BCH formula (see Section~\ref{subsec:modifiedH}) we can argue that the modified Hamiltonian whose time-$\tau$ flow is given by $S_{2n}(\tau)$ is of the form
	\begin{align}
	    \cH_{mod,2n} =  p s -&\epsilon\left(1-q^2\right) s +q 
	     +\tau ^{2n} \Delta \cH_{2n}(q,p,s) +  \cO(\tau^{2n+2}).
	    \label{eq:modvdp3}
	\end{align}
Thus the modified equations read
	\begin{equation}
	    \begin{cases}
	    \dot{q} = s +  \tau^{2n} \frac{\partial \Delta \cH_{2n}}{\partial p} + \cO(\tau^{2n+2})
	    \\
	    \dot{s} = - q -\epsilon(1-q^2) s +  \tau^{2n} \left( -\frac{\partial \Delta \cH_{2n}}{\partial q} - p \frac{\partial \Delta \cH_{2n}}{\partial s} \right) + \cO(\tau^{2n+2})
	    \label{eq:moddynsys}\\ 
	    \dot{p} = -1-p^2+\epsilon\left[(1-q^2)p -2qs\right] +  \tau^{2n} \left( p \frac{\partial \Delta \cH_{2n}}{\partial p} - \Delta \cH_{2n} \right) + \cO(\tau^{2n+2})
	    \end{cases}.
\end{equation}
We know form Proposition \ref{prop:singular} that
the equations for $\dot q$ and $\dot s$ are independent of $p$, and
from Proposition \ref{prop:originequilibrium} that the point $(0,0)$ in the $(q,s)$-plane is an unstable equilibrium of the system. 

If we rewrite the system in polar coordinates on the plane $(q,s)$
with the change of variables $q=r\cos\theta$ and $s=r\sin\theta$,
then the equation for $\dot r$ reads
\begin{equation}
	\dot{r}= \epsilon\, r \sin ^2(\theta ) \left(1-r^2 \cos ^2(\theta )\right) + \tau^{2n} \mathcal{R}_{2n}(r,\theta) + \cO(\tau^{2n+2}) 
	\,,
\end{equation}
Since the modified Hamiltonian is by construction a polynomial in the variables $q$ and $s$, the dependence on $\theta$ of $\mathcal{R}_{2n}$ is only through sums and products of trigonometric functions.
In particular, this implies that the averaged dynamics of $\dot{r}$ obtained by the integration along a period has the form
\begin{equation}
	\frac{1}{2 \pi}  \int_0^{2\pi} \dot{r} \ d\theta   = -\frac{1}{8} r \left(r^2-4\right) \epsilon + \cO(\tau^{2n}).
\end{equation}
One now observes that, modulo high order terms in $\tau$, the stationary points of the averaged dynamics are $r=0$ and $r=2$, which implies that the latter is the radius of the limit cycle, proving the first part of the theorem.

For any fixed order, it is possible to give a more refined estimate of the limit cycle radius by looking at the exact correction from the modified Hamiltonian.

To prove the second part of the statement, we concentrate on the integrator $S_{2}(\tau)(CBABC)$ (since this is the integrator that will be used throughout the simulations in the paper).
The corresponding modified Hamiltonian, in this case, is
\begin{align}
   \cH_{mod,2} =& p s + \epsilon \left(q^2-1\right) s +q  \notag \\
    &+\frac{\tau ^2}{12} \bigg(\big(q^2-1\big) \epsilon ^2 \big(q \big(p q s+q^2+4 s^2-1\big)-p s\big) \notag \\
   &\qquad\quad+\epsilon  \big(p q \big(q^2-2 s^2-1\big)-s \big(-7 q^2+s^2+1\big)\big)-p s+2 q\bigg) + \cO(\tau^3),
\end{align}
leading to the following modified equations for $\dot q$ and $\dot s$
	\begin{equation}
	    \begin{cases}
	    \dot{q} = s+\frac{\tau^2}{12}\left[q \epsilon  \left(q^2-2 s^2-1\right)+\left(q^2-1\right)^2 s \epsilon ^2-s\right]
	    \\
	    \dot{s} = - q -\epsilon(1-q^2) s +\frac{\tau^{2}}{12} \left[-q \left(q^2-1\right) \epsilon ^2 \left(q^2+4 s^2-1\right)+s \epsilon  \left(-7 q^2+s^2+1\right)-2 q \right]
	    \\ 
	    \end{cases},
\end{equation}
and to the radial equation
\begin{align}
	\dot{r} = & \,\epsilon\,r  \sin ^2(\theta ) \left(1-r^2 \cos ^2(\theta )\right) \notag \\
	&+ \frac{\tau ^2}{12} r \bigg(-3 \sin (\theta ) \cos (\theta )-4 r^2 \epsilon ^2 \sin ^3(\theta ) \cos (\theta ) \left(r^2 \cos ^2(\theta )-1\right) \notag \\
	&\qquad\quad+ r^2 \epsilon  \sin ^4(\theta )+\epsilon  \cos ^2(\theta ) \left(r^2 \cos ^2(\theta )-1\right)+\epsilon  \sin ^2(\theta ) \left(1-9 r^2 \cos ^2(\theta )\right)\bigg)
	\,\notag.
\end{align}
An explicit computation then gives
\begin{equation}
	\frac{1}{2 \pi}  \int_0^{2\pi} \dot{r} \ d\theta  =
	-\frac{1}{32} r \epsilon  \left(r^2 \left(\tau ^2+4\right)-16\right) + \cO(\tau^4),
\end{equation}
leading to the claimed radius $r=2-\frac{\tau ^2}{4}+\cO\left(\tau ^4\right)$.
\end{proof}

In the non-stiff regime, we can also perform a perturbative analysis by applying the Poincar\'e-Lindstedt method to study the frequency (and hence the period) of the system (see e.g.~\cite{Andersen1982}).
The first step consists in the time reparametrisation $t'=\omega\,t$,
which leads to the differential equation
\begin{equation}
    \begin{cases}
    \omega q' = X_{\cH_{mod}} q \\
    \omega s' = X_{\cH_{mod}} s\,,
    \end{cases}.
\end{equation}
where the derivatives are now expressed in terms of $t'$, instead of $t$, and, as usual, we omit the decoupled equation for $\dot p$.
Noticing that the modified Hamiltonian vector field depends on the two parameters $\epsilon$ and $\tau$, 
we suppose, in analogy to the traditional approach~\cite{Andersen1982}, 
that all the terms appearing in the equations can be expanded in Taylor series
with respect to such parameters as follows
	\begin{align}
    \omega(\epsilon,\tau) =& \sum_{i=j=0}^{+\infty} \omega_{i,j} \ \epsilon^i \ \tau^{2j}, \label{eq:pertfreq} \\
    q(t,\epsilon,\tau)=& \sum_{i=j=0}^{+\infty} q_{i,j}(t) \ \epsilon^i \ \tau^{2j}, \\
    s(t,\epsilon,\tau)=& \sum_{i=j=0}^{+\infty} s_{i,j}(t) \ \epsilon^i \ \tau^{2j}.
	\end{align}
In particular, notice that we assume all
the expressions to be of even order in $\tau$, given that all the terms appearing in the modified equations are of even order.

For convenience, and without loss of generality, we follow \cite{Andersen1982} and assume that
	\begin{equation}
	\begin{cases}
    q'(0,\epsilon,\tau) = 0,\\
    q(0,\epsilon,\tau) > 0.
	\end{cases}
	\label{eq:phasecondition}
	\end{equation}
This is equivalent to a convenient time shift that simplifies the initial conditions.

The differential equation corresponding to the order $\epsilon^0$, $\tau^0$ then reads
	\begin{equation}
	\begin{cases}
	\omega_{0,0} \ q_{0,0}'(t') = s_{0,0}(t') \\
	\omega_{0,0} \ s_{0,0}'(t') = - q_{0,0}(t')
	\end{cases}
	\end{equation}
whose solution is
\begin{equation}
    \begin{cases}
    q_{0,0}(t)= A \cos \left(\frac{t'}{\omega_{0,0}}\right) + B \sin \left(\frac{t'}{\omega_{0,0}}\right), \\
    s_{0,0}(t)= -A \sin \left(\frac{t'}{\omega_{0,0}}\right) + B \cos \left(\frac{t'}{\omega_{0,0}}\right).
    \end{cases}
\end{equation}
Since we want $q_{0,0}(t)$ and $s_{0,0}(t)$ to have period $2\pi$, this fixes $\omega_{0,0}=1$,
while condition~\eqref{eq:phasecondition} implies $A>0$ and $B=0$.

To fix $A$, we need to consider the order $\epsilon^1$, $\tau^0$, which gives the differential equations
	\begin{equation}
    \begin{cases}
    \omega_{1,0} q_{0,0}'(t') + q_{1,0}'(t')= s_{1,0}(t'), \\
    \omega_{1,0} s_{0,0}'(t')+s_{1,0}'(t') = - q_{1,0}(t') + (1-q^2_{0,0}(t'))s_{0,0}(t')\,.
    \end{cases}\label{eq:10}
	\end{equation}
Inserting the solution of the previous step we can solve~\eqref{eq:10}.
We find that in order to avoid secular behaviours, we need to fix $\omega_{1,0}=0$ and $A=2$.

By repeating this procedure for higher orders of $\epsilon$ and $\tau$, we can compute the matrix $\omega_{i,j}$ and the corresponding solutions.
For instance, up to order $\epsilon^5$ and $\tau^6$, we get
	\begin{equation}
	\label{eq:pertcoeff}
	\omega_{i,j}=\begin{pmatrix}
	 1 & \frac{1}{24} & \frac{3}{640} & \frac{5}{7168} \\
	 0 & 0 & 0 & 0 \\
	 -\frac{1}{16} & \frac{27}{128} & \frac{149}{2048} & \frac{559}{16384} \\
	 0 & 0 & 0 & 0 \\
	 \frac{17}{3072} & \frac{781}{73728} & -\frac{339041}{3538944} & \frac{4695149}{84934656} \\
	 0 & 0 & 0 & 0
	\end{pmatrix}.
	\end{equation}
The first important remark here is that the coefficients of the first row (corresponding to fixing $i=0$ and taking $j=0,1,2,3$ in equation~\eqref{eq:pertfreq}) 
are exactly the same as for the approximation of the frequency obtained by using the modified Hamiltonian (cf.~equation~\eqref{eq:frequency}),
which shows a remarkable consistency between the two methods.
Moreover, equation~\eqref{eq:pertfreq}, with the coefficients $\omega_{i,j}$ given in~\eqref{eq:pertcoeff}, allows us to extend the analytical analysis for the frequency
and period of the limit cycle to the case $\epsilon\neq 0$. 
In Figure~\ref{fig:periodfreq01} we compare the analytical results thus obtained with the numerical results from Section~\ref{subsubsec:nonstiffcasenumerical}. Clearly the match is very accurate,
as the curves are almost indistinguishable, even for very large values of the nonlinear coupling $\epsilon$ and of the time step $\tau$.
\begin{figure}[ht!]
	\centering
	\includegraphics[width=0.8\linewidth]{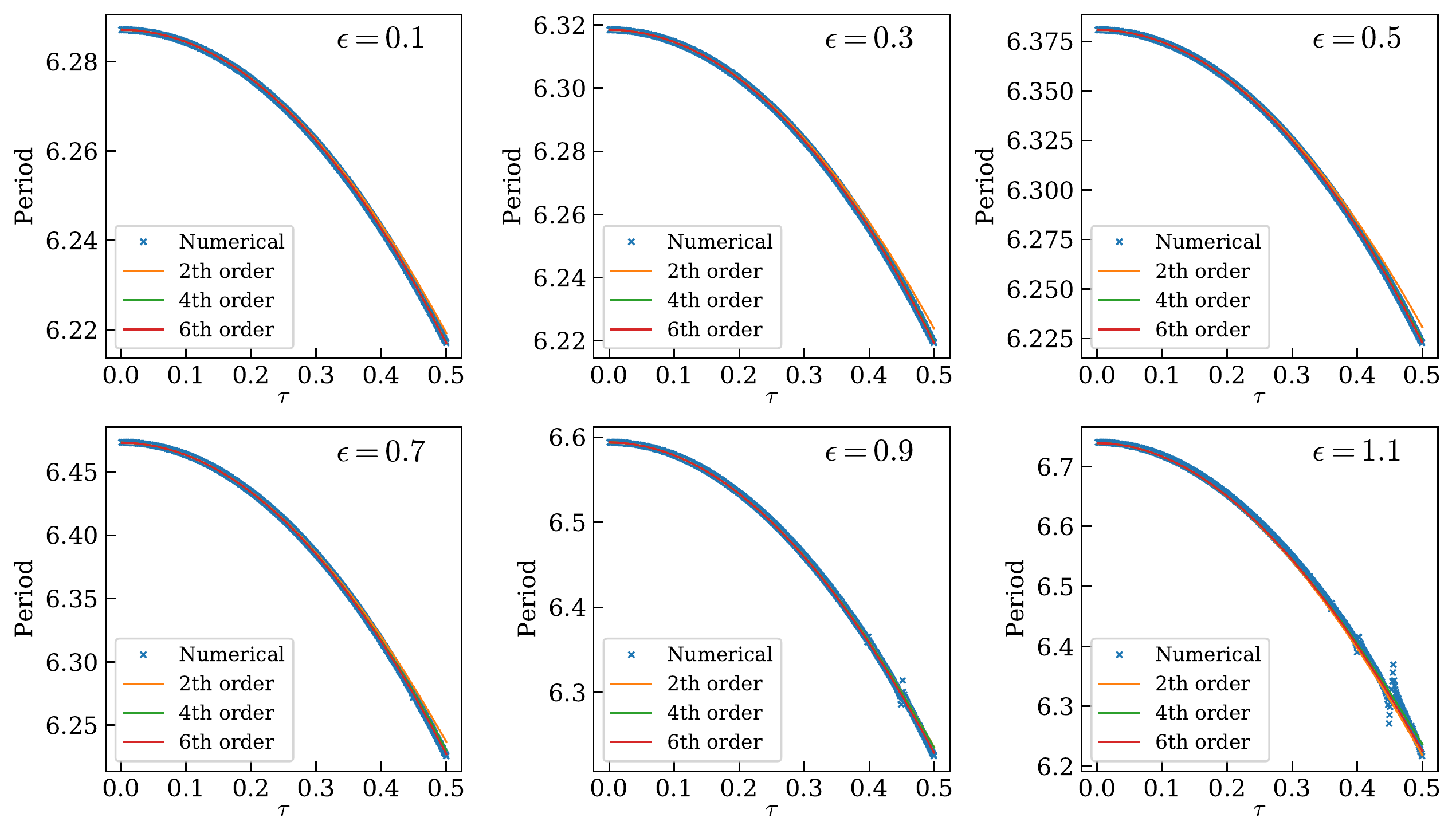}
	\caption{Comparison between the numerical and analytical results (using perturbation theory) for the period of the limit cycle. Each figure is an analogue of Figure~\ref{fig:periodvstaue=0exact} 
	for the value of $\epsilon$ indicated on the top right corner.}
	\label{fig:periodfreq01}
\end{figure}

\subsubsection{$\epsilon \gg 1$ (stiff regime)}

This is allegedly the most difficult regime to study, because $\epsilon$ is large and therefore the nonlinear terms are important.
Typically we must rely on the numerical results.
However, we can give an argument for a reasonable measure of the distance between the simulated numerical dynamics and the original one:
from a direct inspection of the modified Hamiltonian (see e.g.~\eqref{eq:modvdp3}), one can see directly that for any truncation up to 
order $2i$ in $\tau$, we get a polynomial of the same order in $\epsilon$. We formalise this observation in the following result.
\begin{Prop}\label{prop:order}
For any contact splitting integrator of order $2n$ based on the maps~\eqref{eq:vdpmaps},
the truncation at order $2i$ (in $\tau$) of the modified Hamiltonian is a polynomial of degree $2i$ in $\epsilon$.
\end{Prop}
\begin{proof}
It is can be proved that (see~\cite{Bravetti2020}), 
each correction $\Delta \cH_{2i}$ in~\eqref{eq:modH2} is the result of taking ${2i}+1$ nested Jacobi brackets. 
Since the Jacobi bracket is anti-symmetric, we may have at most $2i$ equal terms inside the nested brackets.
Considering that in the splitting~\eqref{eq:hamsplitting} only the $A$ term depends (linearly) on $\epsilon$,
and given the linearity of the Jacobi bracket,
the greatest power in $\epsilon$ is given by the term $\{A,\{A,\{\cdots,\{A,P\}_{\eta}\cdots\}_{\eta}\}_{\eta}\}_{\eta}$, with $P$ being either $B$ or $C$.
We conclude that the maximum degree in $\epsilon$ of $\Delta \cH_{2i}$ is just $2i$.
\end{proof}

From Proposition~\ref{prop:order} it follows that the largest power in $\epsilon$ and $\tau$ in each correction $\Delta \cH_{2i}$ in the modified Hamiltonian
is of the form $(\epsilon\tau)^{2i}$.  Recalling that $\epsilon\gg 1$ in this case, one can expect that to keep the sum~\eqref{eq:modH2} under control, special attention should be given to the size of the product $\epsilon\tau$.
This agrees with the results in Section~\ref{subsubsec:stiffcasenumerical}, where we observed that the limit cycle presents a noticeable deformation for
values of $\epsilon=50,100$ and $\tau=0.01$, or $\epsilon=100$ and $\tau=0.005$, that is, when $\epsilon\tau=0.5,1$.

\section{Geometric numerical integration of forced Li\'enard systems}\label{sec:forcedvdP}

To emphasise the applicability of contact integrators to general Liénard systems, we will now present a brief numerical application of contact integrators to Liénard systems with time--dependent forcing.
As usual, we take the van der Pol oscillator as our benchmark example, and study this system under the influence of a forcing term that is known to give rise to chaotic behavior~\cite{Parlitz1987,Pihajoki2014}.

We stress that this section is meant as an example of possible further applications and the results presented here are by no means meant to be exhaustive analyses or comparisons with the previous literature.
Moreover, we will focus on the numerical aspects and omit the analytical treatment of the modified Hamiltonians: since the computations are analogue to what we have already presented for the unforced van der Pol oscillator, we believe that adding them here would unnecessarily complicate the paper.
\medskip

In the simulations that follow we proceed in analogy to \cite{Pihajoki2014}.
We test the 2nd order contact integrator $S_2(\tau)(CBABC)$ and two different 6th order integrators: 
$S_6^e(\tau)(CBABC)$, with exact coefficients, and $S_6^a(\tau)(CBABC)$, with approximate coefficients taken from family A in Table~\ref{numsol}. 
(these are the integrators that have showed the best performance, cf.~Remark~\ref{rmk:familyA}).
All the comparisons are made with respect to the LSODA solver provided by \texttt{SciPy} \cite{2020SciPy-NMeth} with a relative accuracy parameter of $10^{-13}$ and absolute accuracy parameter of $10^{-15}$.

\subsection{The forced van der Pol oscillator}

Following \cite{Parlitz1987,Pihajoki2014}, we consider a forced van der Pol oscillator of the following form
	\begin{equation}\label{eq:fvdpo}
	    \ddot{x}=\epsilon(1-x^2)\dot{x}-x+A \cos(\omega t),
	\end{equation}
here $A$ is the amplitude of the forcing and $\omega$ its frequency.

Extending \eqref{eq:hamvanderpol} to a time--dependent contact Hamiltonian, we observe that the equation above can be recovered from
\begin{equation}
	\cH(q,p,s,t) = p s -\epsilon (1-q^2) s +q - A \cos(\omega t).
\end{equation}
Indeed, on the $(q,s)$ plane, the corresponding contact Hamiltonian system reduces to
	\begin{equation}
	    \begin{cases}
	    \dot{q} = s \\
	    \dot{s} = \epsilon(1-q^2)s-q+A \cos(\omega t).
	    \end{cases}
	\end{equation}

The nontrivial behaviour of this example is well known \cite{Parlitz1987}: e.g. for the couplings $A=\epsilon=5$ one can show that the system undergoes a bifurcation cascade from a {regular attractor ($\omega=2.457$) to a chaotic one (for $\omega=2.463$)}.

\begin{figure}[ht!]
	\centering
	\includegraphics[width=0.8\linewidth]{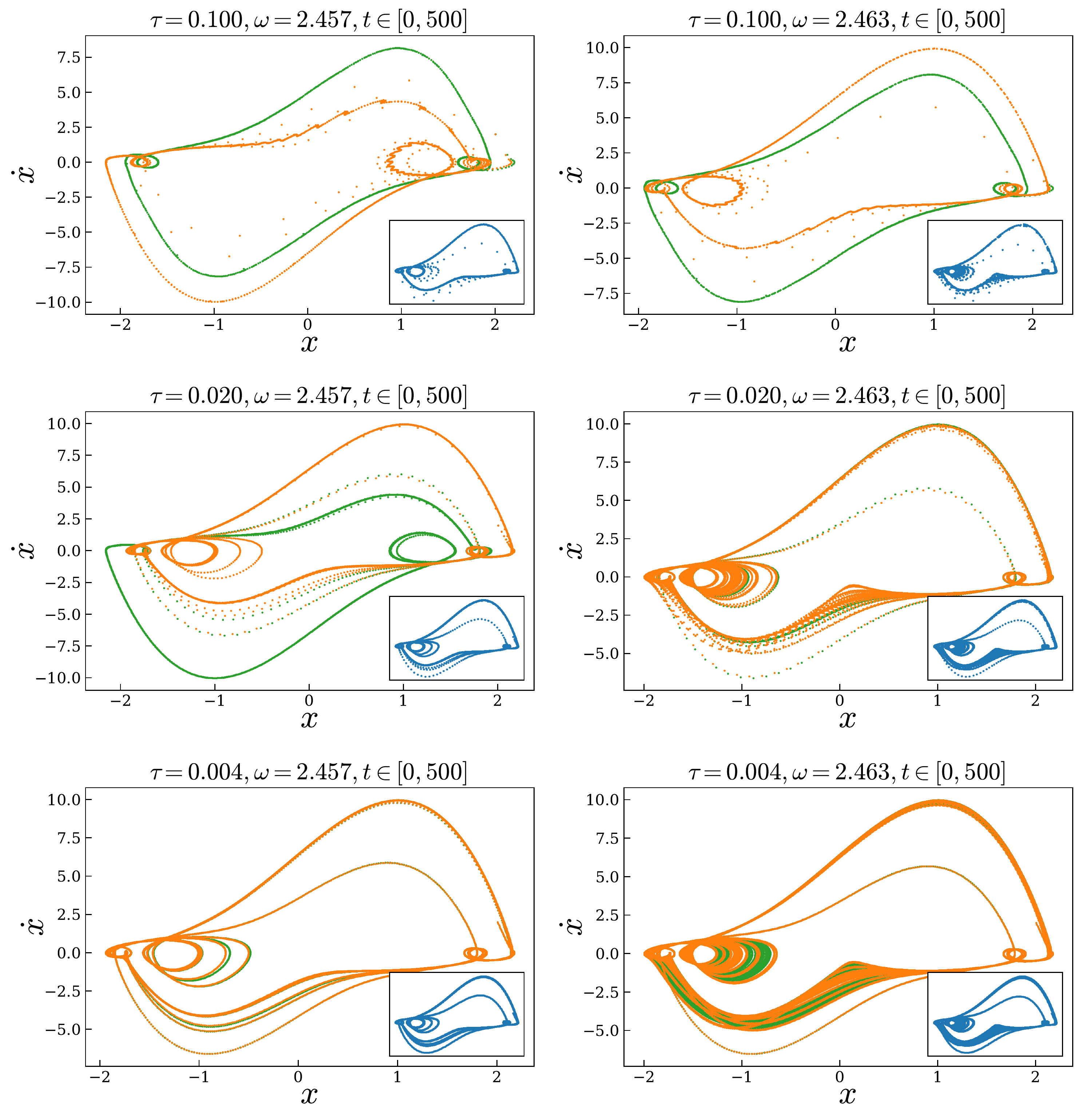}
	\caption{Orbit of the forced van der Pol oscillator with $(x_0, \dot x_0) = (2,2)$. The green dots correspond to the 2nd order integrator and the orange dots to a 6th order approximate integrator (CBABC) with the coefficients taken from family A in Table~\ref{numsol}. 
	Left: regular attractor. Right: strange attractor. From top to bottom the time step is decreasing. The inset plots contain the corresponding trajectory computed with LSODA. it is plotted separately because, besides the first row, it is virtually indistinguishable from the one obtained with the 6th order integrator.}
	\label{fig:bifcascade}
\end{figure}

In the numerical experiments, we propagate the system until $t=500$ and, unless differently specified, the time step is $\tau=0.02$.

As one can see in Figure \ref{fig:bifcascade}, even though we are dealing with a stiff problem, the method is capable of capturing the attractor even for large value of the time step and long integration intervals, rapidly converging to the correct solution as the time--step decreases.

This system in the chaotic regime, $\omega=2.463$, was also the example used to analyze the performances of the modified leapfrog methods introduced in \cite{Pihajoki2014}.
Even though the numerical test in \cite{Pihajoki2014} uses a 6th order integrator, we will still include a test for our second order integrator.

Even though both integrators are geometric in nature, explicit and with fixed time--step, the ones introduced in this paper present two main differences from those in \cite{Pihajoki2014}: they are based on contact geometry instead of symplectic one and they require the integration of only three variables (one of which is the time) instead of six.

In Figure~\ref{fig:orbitforcedvdp} we show the trajectories computed by the aforementioned integrators.
As one can see by comparing Figure~\ref{fig:errorcomparisonforcedvdp} and \cite[Figure 4]{Pihajoki2014}, despite the simplicity of the contact methods, their performance is comparable to the ones presented in \cite{Pihajoki2014}, with the approximate integrator performing better than the exact one: they give results comparable to an established differential equation solver, LSODA, with less computational work: for these simulations the amounts of vector field evaluations of LSODA is $1.4 \cdot 10^6$, while our second order integrator requires $0.1 \cdot 10^6$ evaluations and the sixth order one $0.3 \cdot 10^6$.
These results can also be contrasted with the amount of evaluations for the corresponding algorithms in \cite{Pihajoki2014}, which are $0.7 \cdot 10^6$ (Method 1) and $1.3 \cdot 10^6$ (Method 2).
By avoiding the phase space extension we obtained two concrete advantages: we have reduced both the computational cost of the integrator and the number of possible combinations of the splitting maps.

\begin{figure}[ht!]
	\centering
	\includegraphics[width=0.8\linewidth]{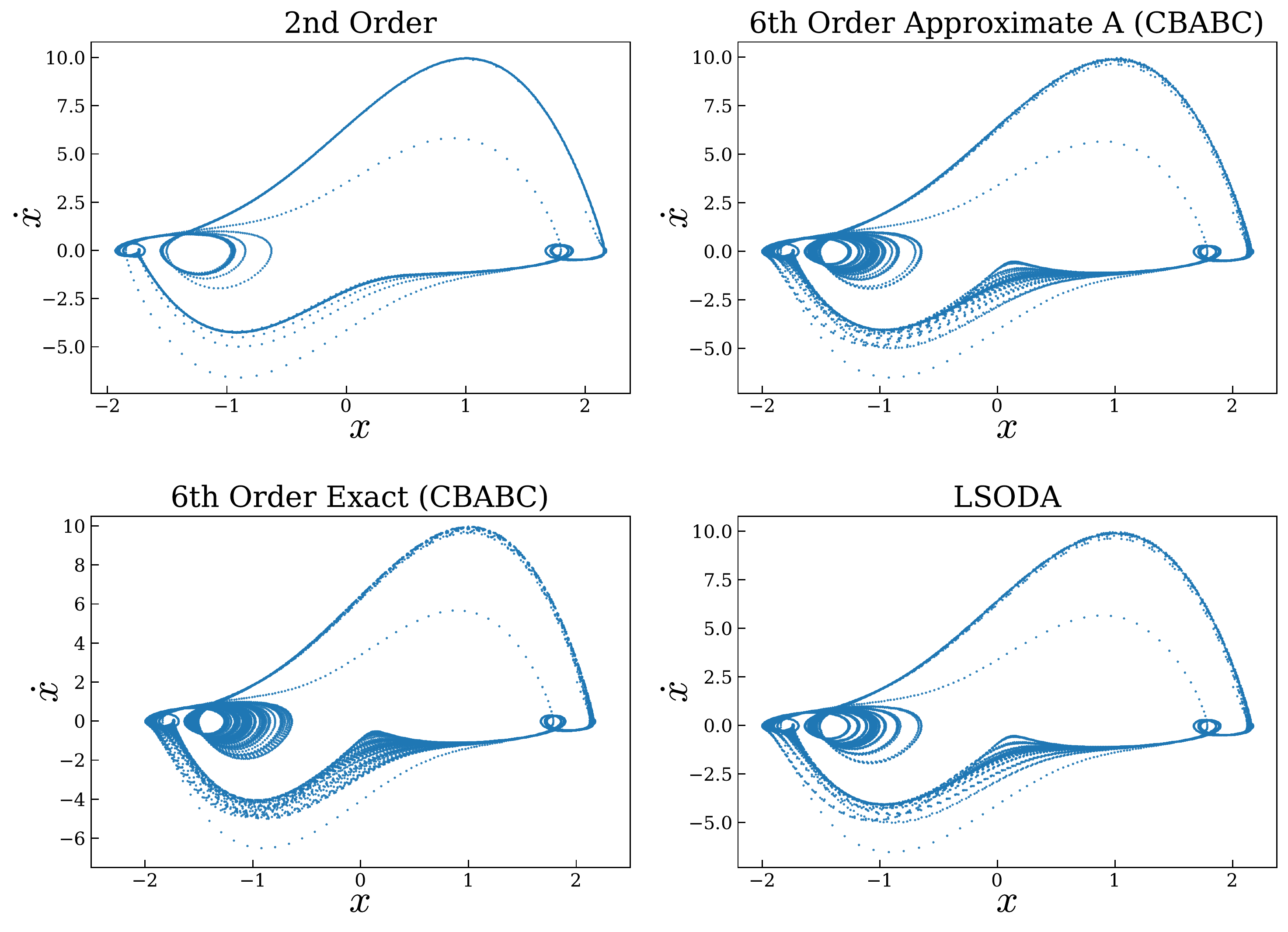}
	\caption{Numerical orbits of the forced van der Pol oscillator \eqref{eq:fvdpo} with $A=\mu=5$, ${\omega=2.463}$ and $(x_0, \dot x_0) = (2,2)$, with the reference contact integrators and LSODA.}
	\label{fig:orbitforcedvdp}
\end{figure}

\begin{figure}[ht!]
	\centering
	\includegraphics[width=0.8\linewidth]{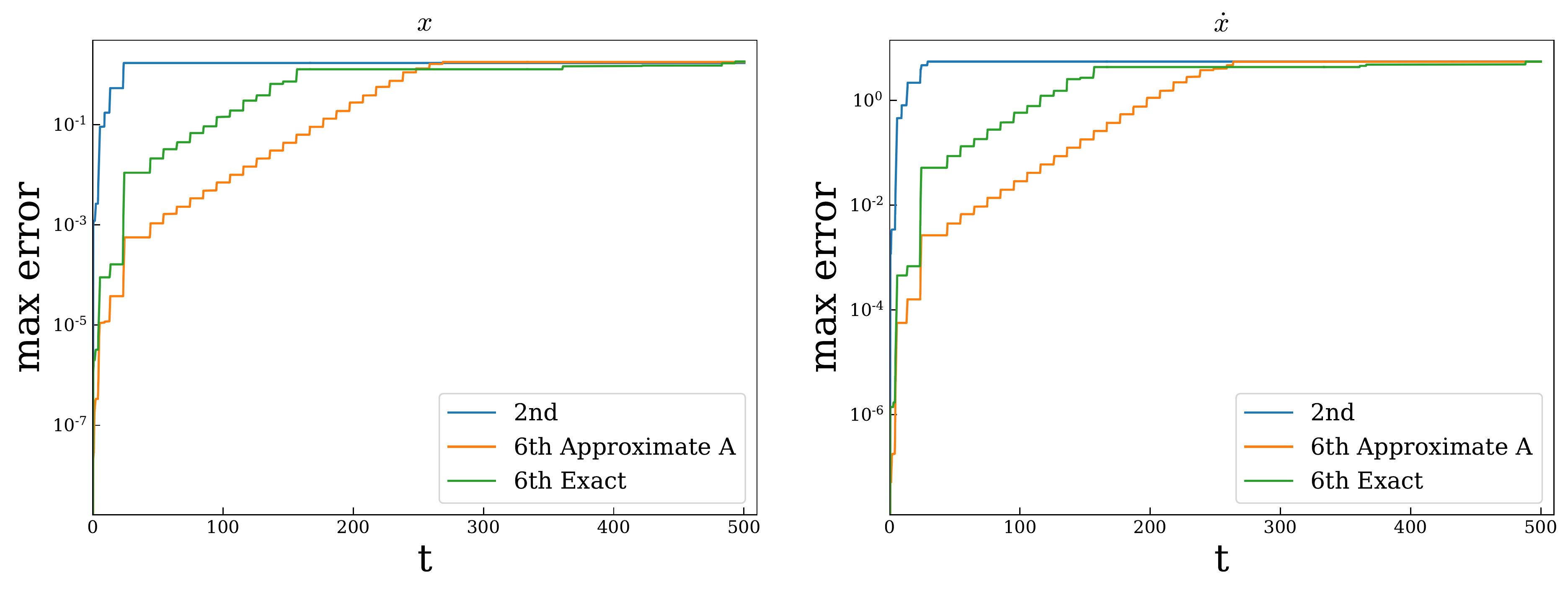}
	\caption{Maximum absolute errors in $x$ and $\dot x$ up to a given time for the reference contact integrators compared to the LSODA method along the orbit in Figure~\ref{fig:orbitforcedvdp}.}
	\label{fig:errorcomparisonforcedvdp}
\end{figure}

\section{Conclusions}\label{sec:conclusions}
In this work we have proposed a novel approach to the geometric numerical integration of an important class of nonlinear dynamical systems, that is, Li\'enard systems.
Such systems are planar systems having a limit cycle, and therefore they cannot be Hamiltonian in the symplectic sense in their original variables.
As a minimal extension, we have considered Li\'enard systems as 2-dimensional projections of contact Hamiltonian systems in three dimensions. 
This Hamiltonisation enables us to use the contact splitting integrators recently introduced in~\cite{Bravetti2020} and therefore to derive a new class of geometric numerical integrators for Li\'enard systems. 
We have used the paradigmatic example of the van der Pol oscillator to show that such formulation can be beneficial both for obtaining accurate numerical integrations of the dynamics at relatively small computational cost, and for deriving complementary analytical results, based on the use of the modified Hamiltonian and modified equations.

Although we have shown here some important results, several questions still remain to be addressed. For instance, we have not fully exploited the modified
Hamiltonian and modified equations in the stiff case; we have not considered further theoretical properties related to the existence of a Hamiltonian structure,
such as e.g.~the preservation of volumes in the 3-dimensional manifold, or the associated Lagrangian structure.
In this context, we remark that the approach investigated here is based on the simplest possible Hamiltonisation of Li\'enard systems by means of contact Hamiltonian systems,
which is obtained by a Hamiltonian that is linear (hence singular) in the momenta. Therefore to derive an associated Lagrangian structure one would have
to use the algorithm for singular contact Hamiltonian systems developed in~\cite{de2019singular}. From the numerical perspective, this could open the door to the use of 
contact variational integrators~\cite{vermeeren2019contact,simoes2020geometry}.
Moreover, other (contact) Hamiltonisations of Li\'enard and spiking systems might be possible, perhaps using non-standard contact structures, and therefore future work should also focus on alternative constructions.

\subsection*{Acknowledgements}
The authors would like to thank Qihuai Liu, Arjan Van der Schaft and Mats Vermeeren for multiple interesting discussions and useful comments, and Mr. Edoardo Zadra for providing emergency computational facilities during the pandemic.
This research was partially supported by the second author's starter grant and by NWO Visitor Travel Grant 040.11.698 that sponsored the visit of AB at the Bernoulli Institute. M. Seri research is partially supported by the NWO co-fund grant 613.009.10.

%\subsection*{Conflict of interest} The authors declare that they have no conflict of interest.

\printbibliography
\end{document}